\documentclass[12pt]{article}
\usepackage{graphicx}
\usepackage{latexsym}
\usepackage{amsmath}
\usepackage{verbatim}
\usepackage{amssymb}
\usepackage{mathrsfs}
\usepackage{enumitem}
\usepackage{enumitem}
\usepackage{amsmath,amssymb}
\DeclareMathAlphabet      {\mathbf}{OT1}{cmr}{bx}{n}
\DeclareFontFamily{OT1}{pzc}{}
\DeclareFontShape{OT1}{pzc}{m}{it}%
{<-> s * [1.15] pzcmi7t}{}
\DeclareMathAlphabet{\mathpzc}{OT1}{pzc}{m}{it}
\usepackage{tikz-cd}
\usepackage{ dsfont }
\usepackage{ mathrsfs }

\usepackage{amsbsy }
\usepackage{graphics}

\usepackage{amsthm}
\newtheorem{thm}{Theorem}
\newtheorem{theorem}{Theorem}[section]

\newtheorem{lemma}[theorem]{Lemma}
\newtheorem{corollary}[theorem]{Corollary}

\newtheorem{definition}[theorem]{Definition}

\newtheorem{prop}[theorem]{Proposition}

\newtheorem{remark}{Remark}

\newcommand\mathitem{\item\leavevmode\vspace*{-\dimexpr\baselineskip+\abovedisplayskip\relax}}
\usepackage{romannum}
\usepackage[toc,page,title,titletoc,header]{appendix}
\usepackage{hyperref}

\usepackage[all,cmtip]{xy}

\DeclareMathOperator*\uplim{\overline{lim}}
\usepackage [latin1]{inputenc}

\usepackage{abstract}

\usepackage{titling}
\usepackage{blindtext}

\begin{document} \small
	\pagenumbering{arabic}

	\title{Local behaviour of sequences of  three--dimensional generalised monopoles}
	\author{{Guanheng Chen}}
	\date{}
	\maketitle

	

	\begin{abstract}
		The purpose of this paper is to  study the behaviour  of  sequences   of   generalised  monopoles with a uniform  bound on a certain $L^2$--norm.   We focus on the  case  that  the target  hyperK\"{a}hler manifolds   are Swann bundles. 
		In 3--dimensional case, suppose that there exists an open submanifold $Y'$ such that 
		the hyperK\"{a}hler potential along the monopoles has a uniform lower bound over $Y'$. Then we show that  there exist  convergent subsequences  of  generalised monopoles over any compact subset of $Y'$.  Under  similar assumptions,  the same conclusion holds for the generalised harmonic spinors in dimension four.
	\end{abstract}
	
	\section{Introduction and Main results}
	Given a Riemannian manifold $(Z, g_Z)$ of dimension three or four,    fix a $Spin^c$ structure $Q\to Z$.  The Seiberg--Witten equations are    non--linear PDE  associated with this $Spin^c$ structure.   A solution to the   equations consists  of a $Spin^c$ connection and a section  of an associated bundle of $Q$.   The latter is called a  \emph{spinor}.  Note that the typical fiber of the associated bundle  is  the quaternionic space  $\mathbb{H}$.

	In \cite{T0} and \cite{VYP}, C. H. Taubes  and V. Y. Pidstrygach generalise    these equations to more general settings. These generalisations are called \emph{generalised Seiberg--Witten equations}. We refer to their solutions   as \emph{(generalised) monopoles}.  The idea behind the generalisation is shown in  the following two aspects.  Firstly, the $Spin^c$ group is replaced by a more general Lie group  called the $Spin^G$ group. Here $G$ can be any fixed compact Lie group. Secondly,  the typical  fiber $\mathbb{H}$ of the  associated bundle    is replaced by  a hyperK\"{a}hler  manifold $(M, g_M, I_1, I_2, I_3)$ with   certain symmetries.  Such a manifold is called a \emph{target manifold}. In particular, $u$ is a section of a fiber bundle and the fiber is not necessarily   a vector space. Many well--known equations in gauge theory are just special cases of the generalised Seiberg--Witten equations. (See Section 5.2 of \cite{MC1}.)  

	Compactness property  is  one of the central  problems  when  we study the moduli space  of non--linear equations. In the case that  the target manifold  is an $\mathbb{H}$--module, many compactness  results have been achieved by many authors, e.g. \cite{T2},  \cite{HW}, and \cite{WZ}.    Their methods are based on Taubes's techniques  in \cite{T1}. For the general cases,      the generalised Seiberg--Witten  equations are highly non--linear. Other aspects of the  generalised Seiberg--Witten equations  have been studied by   many  mathematicians as well, eg. \cite{AH}, \cite{BT},  \cite{VT}, etc.

	The purpose of this paper is to  study the behaviour  of   sequences  of    generalised  monopoles with  a uniform bound on a  certain   $L^2$--norm.  Such a  $L^2$--norm is  an analogue   of the  usual $L^2$--norm of spinors. We focus on the  cases  that  the target manifold  belongs to a class of  hyperK\"{a}hler manifolds with  hyperK\"{a}hler  potentials  called Swann bundles \cite{AS}. Roughly speaking, a Swann bundle $\mathcal{U}(N)$ is a fiber bundle over a quaternionic K\"{a}hler manifold $N$ and the fiber is of the form $\mathbb{H}^{\times} /\mathbb{Z}_2$. \textbf{We assume that $N$ is compact; unless otherwise stated.}
	Alternatively, a  Swann bundle can be viewed as  a metric cone of a 3--Sasaki manifold. Given a sequence of 3--dimensional generalised monopoles, suppose that the images  of the spinor part of the monopoles stay away from the singularity of the metric cone, then
	we show that there exists  a convergent subsequence.    Besides, the methods can be applied to 4--dimensional cases  equally well whenever the Lie group $G$ is zero--dimensional. In these cases, the generalised Seiberg--Witten equations are  reduced to the generalised Dirac equation whose  solutions  are  called \emph{harmonic spinors}.

	The main results of this paper  are summarised as follows:
	\begin{thm} \label{thm0}
		Let $(Y, g_Y)$ be a Riemannian 3--manifold with a $Spin^G$ structure $\pi: Q \to Y$. Suppose that the  target manifold  $(M, g_M, I_1, I_2, I _3) $ is a    Swann bundle $\mathcal{U}(N)$ over  a compact
		quaternionic $K\ddot{a}hler$ manifold $N$. Let $\rho_0$ be the unique  hyperK\"{a}hler  potential defined in Section \ref{section2}.
		
		Let $\{(A_n, u_n)\}_{n=1}^{\infty}$ be a sequence of  generalised monopoles (see Equations (\ref{eq6})) defined by   $\pi: Q \to Y$ and  the Swann bundle. Suppose that $\{(A_n, u_n)\}_{n=1}^{\infty}$ has a  uniform bound on the $L^2$--norm, i.e.,
		$\int_{Y} \rho_0 \circ u_n \le  c_{\heartsuit}$.  Assume that $Y' \subset Y $  is  an open   submanifold of  $Y$ such that     $ \uplim\limits_{n \to \infty} \inf\limits_{Y'}\rho_0 \circ u_n  \ge  c_{\diamondsuit}^{-1}$ for some large constant $c_{\diamondsuit}>0$. Then after gauge transformations, there exists a subsequence of  $\{(A_n, u_n)\}_{n=1}^{\infty}$ that  converges  in $C^{\infty}_{loc}$ to a monopole  $(A, u)$ over $Y'$.
	\end{thm}

	\begin{thm} \label{thm1}
		Let $(X, g_X)$ be a Riemannian 4--manifold with a $Spin^G$ structure $\pi: Q \to X$. Suppose that the Lie group $G$ is zero--dimensional.   Let    $(M, g_M, I_1, I_2, I _3) $ and $\rho_0$  be the Swann bundle  and   the  hyperK\"{a}hler  potential respectively as in Theorem \ref{thm0}.
		
		Let $\{u_n\}_{n=1}^{\infty}$ be a sequence of harmonic spinors (see Definition \ref{def2}) defined by  $\pi: Q \to X$ and the  Swann bundle. Suppose that    $\int_{X} \rho_0 \circ u_n \le c_{\heartsuit}$ and  $X' \subset X $ is  an open   submanifold of $X$  such that     $ \uplim\limits_{n \to \infty} \inf\limits_{X'}\rho_0 \circ u_n  \ge c_{\diamondsuit}^{-1}$ for some constant $c_{\diamondsuit}>0$. Then there exists a subsequence of  $\{u_n\}_{n=1}^{\infty}$ that  converges  in $C^{\infty}_{loc}$ to a  harmonic spinor $u$ over $X'$. 
	\end{thm}

	\begin{remark}
		Under assumptions $\int \rho_0 \circ u_n \le c_{\heartsuit}$ and   $ \uplim\limits_{n \to \infty} \inf\limits_{Y'}\rho_0 \circ u_n  \ge c_{\diamondsuit}^{-1}$,  we can   deduce a uniform bound on the $L^{\infty}$--norm of  the curvatures and  the $W^{1,2}$--norm of  the spinors. One may apply the Uhlenbeck compactness to get a convergent subsequence of the connections. However, the standard elliptic bootstrapping argument cannot extract  a convergent  subsequence of the harmonic spinors, because the generalised Dirac operator is highly non--linear. The constants that appear in the elliptic estimates are not uniform. (See pages 55--56 of \cite{HS}.)
	\end{remark}
	
	Our methods here are  inspired by  \cite{TW}   and   we  also make use of  the monotonicity property of the frequency function as in \cite{HW}.  A brief summary  of the proof is given in Section \ref{section3}. 
	
	\begin{remark}
		For an arbitrary sequence of monopoles $\{(A_n, u_n)\}_{n=1}^{\infty}$ there may be no $Y'$ ($X'$) as in Theorem \ref{thm0} (\ref{thm1}).  For example, if $ {\uplim\limits_{n\to \infty}}\int  \rho_0 \circ u_n =0$, then we cannot find such  submanifolds  $Y'$  and  $X'$.
		
		The existence of the  submanifolds  $Y'$  and  $X'$ is highly non--trivial to verify. Here is a possible way to    detect the existence of  $Y'$ and $X'$.      One can define  closed sets $S_{n,k} =\{x \in Y, \vert \rho_0 \circ u_n(x) \le \frac{1}{k} \}$. Then for any fixed $k$, $\{S_{n,k}\}_{n=1}^{\infty}$ converges to a closed set $S_k$ in Hausdorff distance. 
		Define  a set $S: = \cap_{k=1}^{\infty}S_k$. If $S \ne Y$, then there exists $k_0$ such that $S_{k_0} \ne Y$. As a result, the closed set $S_{k_0}^{\epsilon} =\{x \in Y \vert dist(x, S_{k_0}) \le \epsilon\}$ is not equal to $Y$ for sufficiently small $\epsilon>0$. By the definition of   the   Hausdorff distance, any open submanifold  $Y'   \subset Y-S_{k_0}^{\epsilon}$ lies inside  $ Y-S_{n, k_0}$ for sufficiently large $n$. In particular, $\rho_0 \circ u_n(y) \ge {k_0}^{-1} $ for any $y \in Y'$. 
		
		Another possible way is to follow the idea in \cite{HW}, \cite{T1}, \cite{T2}.  One can define an  $L^{\infty}$--function $\rho$ by the rule  that $\rho(x)=\uplim\limits_{n \to \infty} \rho_{0} \circ u_n(x)$. 
		Similar to Proposition 2.1 of \cite{T1},  $\{\rho_{0} \circ u_n\}_{n=1}^{\infty}$ converges   weakly to $\rho$ in the  $W^{1, 2}$--topology and strongly in  the $L^p$--topology ($1 \le p< \infty$).  The proof of Proposition 2.1 of \cite{T1} only relies on a priori estimates and the $Weitzenb\ddot{o}ck$  formula.  The relevant estimates  and  the $Weitzenb\ddot{o}ck$  formula are true in our setting (see Section \ref{section4}). Thus
		the argument can be applied to our case equally well and  we do not repeat them here. Suppose that $c_{\heartsuit}^{-1} \le \int_Y \rho_0 \circ u_n  \le c_{\heartsuit}$ for some $c_{\heartsuit}>0$. If one can show that $\rho$ is continuous, then  the $L^p$--convergence implies that  $Y-\rho^{-1}(0)$ is a non--empty open set.
		

		Unfortunately, at present,  we can't find a sufficient condition to guarantee that $S \ne Y$ or $\rho$ is continuous.  It is interesting to study the  properties of $\rho$ and $S$ in the future.
	\end{remark}

	\section{Preliminaries}
	In this section, we review some  essential definitions  based on the papers \cite{MC}, \cite{MC1}, and \cite{HS}.
	\subsection{ $Spin^G$ group  and $Spin^G$ structure}
	Let $G$ be a compact Lie group and $\varepsilon$ be  a central element of $G$ satisfying $\varepsilon^2=1$. The  $Spin^G$ group is defined by $$Spin^G_{\varepsilon} (n): =\frac{Spin(n) \times G}{<(-1, \varepsilon)>}. $$
	Here are some examples: $Spin^{U(1)}_{-1}(n)=Spin^c(n)$,  $Spin^{
		\mathbb{Z}_2}_{-1}(n) =Spin(n)$,  $Spin^{\{e\}}_{-1}(n)=SO(n)$ and $Spin^G_1(n) =SO(n) \times G$.
	Similar to the $Spin^c$ case, we have the following  exact sequence:
	\begin{equation} \label{eq18}
		1 \to <(1, \varepsilon )>\to Spin_{\varepsilon}^G(n) \stackrel{\phi_0}{\longrightarrow}  SO(n) \times \bar{G} \to 1,
	\end{equation}
	where $\bar{G} = \frac{G}{\{1, \varepsilon\}}$.

	Let $Sp(1)$ denote the unit sphere in $\mathbb{H}$. We have isomorphisms $Spin(3)\cong  Sp(1)$ and  $Spin(4)\cong  Sp(1) \times Sp(1)$. To distinguish the copies of $Sp(1)$ in $Spin(4)$, we  denote the first one by $Sp(1)_+$ and the second one by $Sp(1)_-$.
	
	We focus on the cases that $n=3$ or $4$ and denote  the  $Spin^G_{\varepsilon} (n)$ group by $H$ throughout.
	
	\begin{definition}
		Let $P_{SO(n)}$ be the frame bundle of a Riemannian manifold $(Z, g_Z)$.  Let  $P_{\bar{G}} \to Z$ be a principal  $\bar{G}$--bundle.  A $Spin_{\varepsilon}^G(n)$ structure is a principal   $Spin_{\varepsilon}^G(n)$--bundle $Q  \to Z$  together with   a   covering     $\phi: Q \to  P_{SO(n)} \times_Z P_{\bar{G}}$ such that $\phi(pg)=\phi(p)\phi_0(g) $ for any $p\in Q $ and $g\in Spin^G_{\varepsilon}(n)$, where $\phi_0: Spin_{\varepsilon}^G(n) \to SO(n) \times \bar{G} $ is the  covering in (\ref{eq18}).
		
	\end{definition}
	
	\subsection{Target hyperK\"{a}hler  manifold  }
	A \emph{hyperK\"{a}hler   manifold} is a Riemannian manifold $(M, g_M)$  endowed with a triple of complex structures $\{I_i\}_{i=1}^3$ which are covariantly constant with respect to the Levi--Civita  connection $\nabla^M$ and   satisfy  the quaternionic relations.  Note that  the tangent bundle $TM$  is  a  bundle of $\mathbb{H}$--module because we have the following ring homomorphism
	\begin{equation}
		\begin{split}
			I: &\mathbb{H} \to End{TM}\\
			&\zeta=h_0 + h_1 i +h_2 j + h_3 k \to I_{\zeta} =h_0 Id_{TM} + h_1 I_1 + h_2 I_2 + h_3 I_3.
		\end{split}
	\end{equation}
	Observe that $I_{\zeta}$ is still a complex structure whenever $|\zeta|=1$ and $  \zeta  \in Im \mathbb{H}$.
	Under the isomorphism  $\mathfrak{sp}(1) \cong Im \mathbb{H}$, the hyperK\"{a}hler  form $\omega \in \Omega^2(M, \mathfrak{sp}(1)^*)$ is defined by $\omega_{\zeta} =g_M( \cdot, I_{\zeta}\cdot)$ for any $\zeta \in \mathfrak{sp}(1) \cong  Im\mathbb{H}  $.

	\subsubsection{Actions on \emph{hyperK\"{a}hler manifolds}}\label{section1}
	\begin{definition}
		Let $(M, g_M,  I_1, I_2, I_3)$ be a hyperK\"{a}hler   manifold with an isometric $Sp(1)$--action.  The $Sp(1)$--action is called permuting if
		\begin{equation*}
			dq I_{\zeta} dq^{-1}= I_{q \zeta \bar{q}},
		\end{equation*}
		for any $q \in Sp(1)$, where $dq: T_xM \to T_{qx} M$ is the differential of this action   and $\zeta \in Im \mathbb{H}$ with $|\zeta|=1$.
		
	\end{definition}

	Assume that $M$ admits a \emph{hyperK\"{a}hler  $G$--action}, i.e.,  it preserves the metric and complex structures.   The hyperK\"{a}hler  $G$--action is called \emph{hyperHamiltonian} if   it is Hamiltonian with respect to each  $\omega_{\zeta}$.    Then there  is a $G$--equivariant map  $\mu: M \to \mathfrak{sp}(1)^* \otimes \mathfrak{g}^*$  such that
	\begin{equation}
		\begin{split}
			\iota_{K^{M, G}_{\xi} } \omega_{\zeta}= <d \mu, \zeta \otimes \xi>,
		\end{split}
	\end{equation}
	where $\zeta \otimes \xi \in \mathfrak{sp}(1) \otimes \mathfrak{g} $ and $K^{M, G}_{\xi}\vert_m :=\frac{d}{dt} exp(t \xi) \cdot m \vert_{t=0}$ is the fundamental vector field.
	Such a map $\mu$ is called  a \emph{hyperK\"{a}hler  moment map}.
	
	Assume that    $(M, g_M,   I_1, I_2, I_3)$ admits a permuting $Sp(1)$--action and a   hyperK\"{a}hler $G$--action.   Suppose  that the   $Sp(1)$--action and the   hyperK\"{a}hler $G$--action  satisfy the following assumptions:
	\begin{enumerate}
		\item
		The $G$--action commutes with the $Sp(1)$--action.
		\item
		Let  $\varepsilon$ be  a central element of $G$ satisfying $\varepsilon^2=1$. We require  that the element $(-1, \varepsilon ) \in  Sp(1) \times G$ acts trivially on $M$.
	\end{enumerate}
	Under these assumptions,  we have an  $Sp(1) \times G$--action on $M$ and it descends to a   $Spin^G_{\varepsilon} (3)=\frac{Sp(1) \times G}{<(-1, \varepsilon)>} $--action.  This action is called a \emph{permuting $Spin^G_{\varepsilon} (3)$--action}.  A $Spin^G_{\varepsilon} (4)= \frac{Sp(1)_+ \times Sp(1)_- \times G}{<(1, \varepsilon)>}$--action is called  permuting  if  the $Sp(1)_+$--action is permuting while the $Sp(1)_- \times G$--action is hyperK\"{a}hler.


	
	\subsubsection{HyperK\"{a}hler potenial} \label{section2}
	Let  $(M, g_M,   I_1, I_2, I_3)$  be a hyperK\"{a}hler manifold. 
	A function $\rho \in C^{\infty}(M, \mathbb{R})$ is called \emph{a hyperK\"{a}hler   potential} if  it satisfies $dI_{\zeta} d \rho=2 \omega_{\zeta}$ for any $\zeta \in \mathfrak{sp}(1)$ with $|\zeta|=1$.     Here $I_{\zeta}$ acts on $d\rho$ by $I_{\zeta} d \rho (\cdot) : = d\rho (I_{\zeta}\cdot)$.
	
	Suppose that  $(M, g_M,   I_1, I_2, I_3)$  admits    a permuting   $Sp(1)$--action.  By \cite{VYP} (also see Proposition 2.2.7 of \cite{MC1}),    the hyperK\"{a}hler form $\omega$ is exact. Define a map $\chi: \mathfrak{sp}(1) \otimes  \mathfrak{sp}(1) \to \Gamma(M, TM)$ by
	\begin{equation*}
		\zeta \otimes \zeta' \to -I_{\zeta'} K^{M, Sp(1)}_{\zeta},
	\end{equation*}
	where $K^{M, Sp(1)}_{\zeta}$ is the fundamental vector field of the $Sp(1)$--action.
	Decompose $\chi$ into $\chi_0 + \chi_1 + \chi_2$, where $\chi_0=-\frac{1}{3}\sum_{l=1}^3 I_l  K^{M, Sp(1)}_{\zeta_l}$ is its diagonal,  $\chi_1$ is its antisymmetric part and $\chi_2$ is its trace--free symmetric part.  The following lemma gives a sufficient  condition  of the existence of a   hyperK\"{a}hler  potential.

	\begin{lemma}[Proposition 5.5 of \cite{AS}, Lemma 3.2.3 of \cite{HS}]
		Let $M$ be a hyperK\"{a}hler  manifold with a  permuting   $Sp(1)$--action. Suppose that $\chi_2=0$. Then there is a unique hyperK\"{a}hler  potential  $\rho_0$ such that
		$\rho_0 = \frac{1}{2} g_M(\chi_0, \chi_0). $ 
		
		Conversely, if $M$ admits a hyperK\"{a}hler potential, then $M$ admits a local permuting $Sp(1)$--action with $\chi_2=0$.
	\end{lemma}
	The choice of a   hyperK\"{a}hler  potential is not unique. We fix the choice provided by the  above lemma throughout.


	\subsection{Swann bundle}
	The hyperK\"{a}hler manifolds admitting  a  permuting   $Sp(1)$--action and a hyperK\"{a}hler  potential are the Swann bundles.  This class of hyperK\"{a}hler manifolds  is constructed by A. Swann \cite{AS}.  We briefly review the construction of Swann bundles and their properties.

	Let $(N, g_N)$ be a  \emph{quaternionic K\"{a}hler manifold} of dimension $4n$, i.e., its holonomy is contained in $Sp(n)Sp(1):=(Sp(n) \times Sp(1)) / (\pm 1 )$.  Let $F$ denote the $Sp(n)Sp(1)$--reduction of the $SO(4n)$ frame bundle of $N$.  Then $\mathscr{C}(N) : = F/Sp(n) $ is a principal $SO(3)$--bundle over $N$.      The  \emph{Swann bundle} is defined by
	\begin{equation}\label{eq15}
		\mathcal{U}(N) : = \mathscr{C}(N) \times_{SO(3)} (\mathbb{H}^{\times} /\mathbb{Z}_2),
	\end{equation}
	where $\mathbb{H}^{\times} :=\mathbb{H}-\{0\}$.
	Suppose that the scalar curvature of $(N, g_N)$  is positive, then  the   Swann bundle is a  hyperK\"{a}hler  manifold with  a permuting $SO(3)$ or $Sp(1)$--action and vanishing $\chi_2$. 
	The metric on $\mathcal{U}(N)$ is  of the form 
	$$g_{\mathcal{U}(N)} =dr^2 +r^2g_{\mathscr{C}(N)},$$  where $r$ is the radius coordinate of $\mathbb{H}^{\times}$. 

	What follows  are two facts related to the desired  properties in Section \ref{section1} and \ref{section2}:
	\begin{enumerate} [label=\textbf{F.\arabic*}]
		\item  \label{a}
		The Swann bundle $\mathcal{U}(N)$ is a  hyperK\"{a}hler  manifold with $\chi_2 =0$. In this case, the hyperK\"{a}hler  potential is $\rho_0 = \frac{1}{2} r^2$.
		
		\item \label{b}
		If a Lie group $G$
		acts on $N$, preserving the quaternionic K\"{a}hler structure, then the action can be lifted to a hyperHamiltonian action of $G$ on $\mathcal{U}(N) $. Also, the action  leaves  $ \mathscr{C}(N)$ invariant.  (See  Proposition 4.2  and Theorem 4.6 of \cite{AS}.)
	\end{enumerate}
	For more details about the  construction of $\mathcal{U}(N)$ and its properties,  we refer the  reader to Swann's original paper \cite{AS}.

	We remark that a class of compact quaternionic K\"{a}hler manifolds with  positive scalar curvature is the so--called    \emph{Wolf spaces}.   The Wolf spaces are the only  compact, homogeneous, quaternionic K\"{a}hler  manifolds. They are classified by Wolf \cite{JW}  and  Alekseevskii   \cite{DVA1}, \cite{DVA2}.  Some examples of Wolf spaces are as follows:
	\begin{equation*}
		\begin{split}
			& \mathbb{HP}^n=\frac{Sp(n+1)}{Sp(n) \times Sp(1)},  Gr_2(\mathbb{C}^n)= \frac{SU(n)}{S(U(n-2) \times U(2))}, \\
			&  \widetilde{Gr}_2(\mathbb{R}^n)= \frac{SO(n)}{S(SO(n-4) \times SO(4))}, \frac{G_2}{SO(4)}.
		\end{split}
	\end{equation*}

	\paragraph {Examples} Even we mainly  use the properties \ref{a}, \ref{b} of Swann bundles in this paper, it is still worth  reviewing   a class of examples $\mathcal{O}$ of Swann bundles.   The Swann bundle $\mathcal{O}$ is obtained by    P. B. Kronheimer  as a  moduli space of Nahm equations \cite{K}. The hyperK\"{a}hler structure on $\mathcal{O}$ is written  down by P. Kobak and Swann explicitly in  \cite{KS}.  The description
	that follows of Swann bundles paraphrases what is presented in  \cite{AS} and \cite{KS}:

	Let $G$ be a compact, simply connected, simple Lie group. Let $G^{\mathbb{C}}$ be its complexification.  Let $\mathfrak{g}$  and  $\mathfrak{g}^{\mathbb{C}}$ denote the Lie algebra of  $G$ and  $G^{\mathbb{C}}$  respectively.  Fix a real structure $\sigma$ on $\mathfrak{g}^{\mathbb{C}}$  such that $\mathfrak{g}$ is the eigenspace to the eigenvalue $1$.  Choose  a Cartan subalgebra $\mathfrak{h}$ of $\mathfrak{g}^{\mathbb{C}}$.   For example,  we could  take $G=SU(2)$,  $G^{\mathbb{C}} =SL(2, \mathbb{C})$, and $\sigma$ to be the minus conjugate transpose of matrices.  The Lie algebra $\mathfrak{sl}(2, \mathbb{C})$ has a  Cartan basis
	\begin{equation*}
		E=\left( \begin{matrix}
			0 & 1\\
			0 & 0\\
		\end{matrix}\right),
		H=\left( \begin{matrix}
			1 & 0\\
			0 & -1\\
		\end{matrix}\right),
		F=\left( \begin{matrix}
			0 & 0\\
			1 & 0\\
		\end{matrix}\right).
	\end{equation*}

	Fix a system of roots  $\Delta$ with positive roots $\Delta_+$.  Let $\alpha \in \Delta_+$ be a highest root.  Then  $\alpha$ induces a Lie algebra embedding $\mathfrak{sl}(2, \mathbb{C})    \hookrightarrow \mathfrak{g}^{\mathbb{C}}$.   The image of $\{E, H, F\}$ is denoted by $\{E_{\alpha}, H_{\alpha}, F_{\alpha}\}$. We choose the embedding such that it is  compatible with $\sigma$, in the sense that $\sigma(E_{\alpha}) =-F_{\alpha}$ and $\sigma(H_{\alpha})=-H_{\alpha}$.    Let $\mathcal{O}$ be the orbit of $E_{\alpha}$ under the adjoint action of $G^{\mathbb{C}}$.  
	The  hyperK\"{a}hler  structure and  the $Sp(1)$--action are given   as follows:
	
	The complex structure $I$ on $G^{\mathbb{C}}$ descends to the orbit $\mathcal{O}$.  Fix a negative Killing form $<, >$ on $\mathfrak{g}^{\mathbb{C}}$.  Define a  function  $\eta: \mathcal{O} \to \mathbb{R}$ by $$\eta \vert_X : = |X|^2 = <X, \sigma X>.$$
	Note that $\eta$ is $G$--invariant.   Let $\rho_0(X):=|	E_{\alpha}|\sqrt{\eta} $.  Then we can define a  K\"{a}hler  metric on $\mathcal{O}$ by
	\begin{equation*}
		\begin{split}
			&\omega_I: =\frac{1}{2} d (\rho'_0 I d\eta)\\
			&g(\xi_A, \xi_B) \vert_X=\omega_I(I\xi_A,  \xi_B) \vert_X=2Re( \rho_0'<\xi_A, \sigma \xi_B> + \rho_0''<\xi_A, \sigma X> <\sigma \xi_B, X>),
		\end{split}
	\end{equation*}
	where $A, B \in \mathfrak{g}^{\mathbb{C}}$, $\xi_A=[A, X] , \xi_B=[B, X]$ is the fundamental vector field at $X$, and $\rho_0'= \frac{d}{d\eta} \rho_0, \rho_0''= \frac{d^2}{d^2\eta} \rho_0$.
	
	On the orbit $\mathcal{O}$, the complex symplectic    form $\omega_c$ of Kirillov, Kostant and Souriau   is given by
	\begin{equation*}
		\omega_c(\xi_A, \xi_B) \vert_X: =<X, [A, B]>.
	\end{equation*}
	Write $\omega_c = \omega_J + i\omega_K$.   The $2$--forms $\omega_J, \omega_K$ are expected to serve  as the other two K\"{a}hler  forms.  Therefore, we  get
	\begin{equation*}
		J\xi_A \vert_X = -2 \rho_0'[X, \sigma \xi_A]-2\rho_0'' <\sigma \xi_A, X> [X, \sigma X]
	\end{equation*}
	from the relation $g(\xi_A, \xi_B) =\omega_J(J\xi_A, \xi_B)$. We can deduce a formula  for $K$ similarly.  They satisfy the quaternionic relation $IJ=K$. Also, the computation in Theorem 5.2 of \cite{KS} shows that $J, K$ are almost complex structures.  By Lemma 2.2 of \cite{NJH} (on page 64), $(\mathcal{O}, g, I, J, K)$ is  a   hyperK\"{a}hler  manifold. Proposition  5.6 in  \cite{AS} implies that $\rho_0$ is a hyperK\"{a}hler potential.

	Define an action of   $\mathbb{H}^{\times}$ on $\mathcal{O}$ that   is generated by vector fields $\nabla \rho_{0}$, $I\nabla \rho_0$, $J\nabla \rho_0$ and  $K\nabla \rho_0$. The argument in Proposition  5.5 of \cite{AS} shows that this  $Sp(1) \subset \mathbb{H}^{\times}$--action  is permuting.  After mod out the  $\mathbb{H}^{\times}$--action, the quotient is the Wolf space.

	\subsection{Clifford multiplication  }
	Recall that we set $H=Spin_{\varepsilon}^G(n)$.  Let $C^{\infty}(X ,Y)^H$ denote the space of $H$--equivariant maps from $X$ to $Y$.   The Clifford multiplication in dimension three is a  $Spin_{\varepsilon}^G(3) $--equivariant homomorphism
	\begin{equation*}
		c_3: C^{\infty}(Q, (\mathbb{R}^3)^* \otimes TM)^H \to  C^{\infty}(Q, TM)^H.
	\end{equation*}
	Under the isomorphism  $(\mathbb{R}^3)^*  \cong Im \mathbb{H}$, the Clifford multiplication  is  $c_3(h \otimes v)  =  I_{\bar{h}} v$.
	
	In 4--dimensional  case, the scalar multiplication
	$$Im\mathbb{H} \to End(TM),  \ h \to I_h$$
	extends to an   $H$--equivariant map $Cl_3 \to End(TM)$, i.e., $TM$ is a $Cl_3$--module. Identify $Cl_4^0$ with $Cl_3$. Define
	$$E: =Cl_4 \otimes_{Cl_4^0} (TM, I_1). $$
	The splitting $Cl_4=Cl_4^0 \oplus Cl_4^1$ induces a decomposition  $E=E^+ \oplus E^-$, where $E^+=Cl^0_4 \otimes_{Cl_4^0} (TM, I_1)$ and $E^-=Cl^1_4 \otimes_{Cl_4^0} (TM, I_1)$.
	They are respectively analogues of the $Spin^c$ bundles $S_+$ and $S_-$.  Each of them is  a copy of $TM$ while admitting   different  $Spin_{\varepsilon}^G(4)$--actions. 
	The Clifford multiplication  is  a  $Spin_{\varepsilon}^G(4)$--equivariant map $c_4: C^{\infty}(Q,  (\mathbb{R}^4)^*)^H \to End(E^+ \oplus E^-)$ defined by
	\begin{equation}
		\begin{split}
			c_4(e_0) =
			\begin{pmatrix}
				0 & -1\\
				1 & 0\\
			\end{pmatrix} \mbox{ \ and \ }
			c_4(e_l)=
			\begin{pmatrix}
				0 & -I_l\\
				-I_l & 0\\
			\end{pmatrix},
		\end{split}
	\end{equation}
	where $\{e_0, \{e_l\}_{l=1}^3\}$ is    the standard basis of $(\mathbb{R}^4)^*$.

	\subsection{Generalised Dirac operator }
	Let $\mathcal{N}=C^{\infty}(Q, M)^H$  denote the space of $H$--equivariant maps from $Q$ to $M$. Note that this  is nothing but just  the  space of sections of   the associated  bundle $\mathcal{M}:=Q \times_H M$.    An element of $\mathcal{N}$ is    referred to as a \emph{spinor}.

	Note that the  Lie algebra of the $Spin^G$ group  can be decomposed  into  $Lie H= \mathfrak{so}(n) \oplus \mathfrak{g}$. Let $\mathcal{A} $ denote  the  space of connections on $\pi: Q \to Z$ whose $\mathfrak{so}(n)$--component is induced by the  Levi--Civita  connection, i.e.,
	$$\mathcal{A} = \{A  \in \mathcal{A}(Q) \vert pr_{\mathfrak{so}(n) }\circ  A = \pi_{SO(n)}^* A_Z  \},$$
	where $A_Z$  denotes the  Levi--Civita connection of $P_{SO(n)} \to Z$.  Then $\mathcal{A}$ is an affine space over  $\Omega^1(Q, \mathfrak{g})_{hor}^H = \Omega^1(Z, \mathfrak{g}_Q) $, where $ \mathfrak{g}_Q$ is the associated bundle $Q \times_{Ad} \mathfrak{g}$.
	
	Fix a  connection  $A \in \mathcal{A}  $.  For any $u \in \mathcal{N}$, define  a  covariant derivative  $d_A: \mathcal{N} \to C^{\infty}(Q, (\mathbb{R}^n)^*  \otimes u^*TM)^H$ by
	\begin{equation} \label{eq34}
		d_A u :=du +K^{M, H}_A \vert_u,
	\end{equation}
	where $K^{M, H}_A \vert_u$ is defined by
	\begin{equation*}
		K^{M, H}_A \vert_u(v) := K^{M, H}_{A(v)} \vert_u
	\end{equation*}
	for any vector field $v \in \Gamma(TQ)$.

	We  also define another  covariant derivative on $C^{\infty}(Q, TM)^H$. Before that, let us recall the definition of  the  \emph{connector}.  Let $\pi: E \to B $ be a vector bundle. For any $(v, w) \in E\times_{B} E$,  its  \emph{vertical lift} is defined by $$vl_E(v, w) : = \frac{d}{dt}(v+ tw) \vert_{t=0} \in TE^{vert},$$ where $TE^{vert} := \ker \pi_*$ is the vertical bundle.    The vertical lift $vl_E$ gives an isomorphism $E \times_{B} E  \cong TE^{vert}$.  Given any connection $\Phi: TE \to TE$, the corresponding connector $K_{\Phi}$ is defined by
	\begin{equation*}
		K_{\Phi} : = pr_2 \circ vl_E^{-1} \circ \Phi,
	\end{equation*}
	where $pr_2: E\times_B  E \to E$ is the projection onto the second factor.   The covariant derivative  $\nabla^{\Phi}$ of $\Phi$ can be expressed as $\nabla^{\Phi} s = K_{\Phi} \circ ds $ for any section $s \in \Gamma(E)$.  For more details, please see  \cite{KMS}.

	Take  $E=TM$ and $\Phi$ to be the   Levi--Civita  connection of $(M,g_M)$.  The  corresponding connector is denoted by $\mathcal{K}$.  We define another covariant derivative by the following formula:
	\begin{equation}
		\begin{split}
			\nabla_A^{TM} :  &C^{\infty}(Q, TM)^H  \to  C^{\infty}(Q,  (\mathbb{R}^n)^* \otimes TM)^H \\
			& \nabla_A^{TM} V : = \mathcal{K} \circ d_A^{TM} V,
		\end{split}
	\end{equation}
	where $d_{A}^{TM} V:=(dV +K^{TM, H}_A \vert_V)$ is defined in the same way as in (\ref{eq34}).  Alternatively, $d_{A, w}^{TM} V= \frac{d}{dt} V(\sigma(t)) \vert_{t=0}$, where $\sigma$ is a horizontal  path in $Q$ such that $\sigma'(0)$ is a horizontal lift of $w \in TZ$ with respect to $A$. 
	When we restrict   our attention  to a spinor $u$, then   $\nabla_A^{TM} $ descends to   a  covariant derivative (still denoted by  $\nabla_A^{TM}$) on the vector bundle $\pi_! u^* TM:= u^* TM /H $ in the usual sense. (See Remark.4.3.3 of \cite{MC1}.)

	Similar to the usual case, the generalised Dirac operator is  defined to be  the composition of  the Clifford multiplication $c_n$ and the covariant derivative, i.e., $D_A u : = c_n \circ d_A u$. In  terms of the normal  coordinates  $\{x^i\}$, the Dirac operator is
	\begin{equation} \label{eq4}
		\begin{split}
			D_Au =   \sum\limits_i c_n(dx^i)d_{A, i}u .
		\end{split}
	\end{equation}

	\begin{definition} \label{def2}
		A spinor  $u \in \mathcal{N}$   such that $D_A u=0$ is called a harmonic spinor.
	\end{definition}

	Since the Dirac operator $D_A$ is highly non--linear, it is useful to consider its linearization $D_A^{lin, u}:C^{\infty}(Q, u^*TM)^H \to C^{\infty}(Q, u^*TM)^H$. By Lemma 3.6.8 of \cite{MC}, we know that the linearization of $D_A$ at $u$ is
	\begin{equation}   \label{eq28}
		\begin{split}
			D_A^{lin, u}v =    \sum\limits_i c_n(dx^i) \nabla^{TM}_{A, i} v.
		\end{split}
	\end{equation}

	\subsection{Generalised Seiberg--Witten equations}
	Let  $(M, g_M, I_1, I_2, I _3) $ be a hyperK\"{a}hler  manifold with a permuting $Spin^G$--action.  Let  $\mu: M \to \mathfrak{sp}(1)^* \otimes \mathfrak{g}^*$ be the moment map.
	
	Let $(Y, g_Y) $ be a closed Riemannian 3--manifold and $Q \to Y $ be a $Spin^G$ structure.  
	Let  $\{e_l \}_{l=1}^3$ and $\{\zeta_l\}_{l=1}^3$ be respectively  orthonormal bases  of $(\mathbb{R}^3)^*$  and $\mathfrak{sp}(1)^*$. Then we define  an isomorphism    $(\mathbb{R}^3)^* \cong \mathfrak{sp}(1)^*$ by identifying  $e_l$  with $\zeta_l$.   Under this isomorphism, the moment map induces a map $\mu  : \mathcal{N} \to   C^{\infty}(Q,  (\mathbb{R}^3)^* \otimes \mathfrak{g})^H$. Here we abuse the same notation to denote the induced map. As a result, $\mu(u)$ is a $\mathfrak{g}$--valued 1--form on $Y$ for $u \in  \mathcal{N}$.

	The  3--dimensional generalised Seiberg--Witten equations  ask that a  pair $(A, u) \in \mathcal{A} \times \mathcal{N}$ obey
	\begin{equation} \label{eq6}
		\begin{cases}
			D_A u=0,\\
			*F_a  +\mu(u)=0,
		\end{cases}
	\end{equation}
	where $a$ is the $\mathfrak{g}$--component of $A$ and $F_a$ is its curvature.

	For 4--dimensional case, let $(X, g_X)$ be a closed Riemannian 4--manifold with     a $Spin^G$ structure $Q \to X $. Similar to the  3--dimensional case, we fix   orthonormal bases   $\{\eta_l \}_{l=1}^3$ and $\{\zeta_l\}_{l=1}^3$ for   $\Lambda^{2+}(\mathbb{R}^4)^*$  and $\mathfrak{sp}(1)^*$ respectively. Then we get  an isomorphism    $\Lambda^{2+}(\mathbb{R}^4)^* \cong \mathfrak{sp}(1)^*$ via identifying  these bases. Under this isomorphism, we regard   $\mu$ as a map $\mu : \mathcal{N} \to C^{\infty}(Q,  \Lambda^{2+}(\mathbb{R}^{4})^* \otimes \mathfrak{g})^H $. 
	Therefore, $\mu(u)$ is a $\mathfrak{g}$--valued self--dual 2--form on $X$ for $u \in  \mathcal{N}$.
	The 4--dimensional generalised Seiberg--Witten equations are
	\begin{equation}  \label{eq20}
		\begin{cases}
			D_A u=0,\\
			F_a^+  +\mu(u)=0.
		\end{cases}
	\end{equation}
	
	In both 3--dimensional and 4--dimensional cases, a solution $(A, u)$ of the  generalised Seiberg--Witten equations  is referred to as \emph {a (generalised) monopole.} Let $\mathcal{G}: =C^{\infty}(Q, G)^H$ be the gauge group. For $g \in \mathcal{G}$, the gauge action on $\mathcal{A} \times \mathcal{N}$ is given by
	\begin{equation*}
		\begin{split}
			g \cdot (A, u) =(g^*A, g^{-1} u) = (Ad_{g^{-1}}(A) + g^*\eta, g^{-1}u),
		\end{split}
	\end{equation*}
	where $\eta \in \Omega^1(G, \mathfrak{g})^G$ is the left--invariant Maurer--Cartan form on $G$.
	By  Proposition 4.2.8 of \cite{MC}, the generalised Seiberg--Witten equations are  gauge--invariant.

	\subsection{Sobolev norm}
	For a pair $(A ,u ) \in \mathcal{A} \times \mathcal{N}$, we can define its Sobolev norm  as follows. Fix a smooth reference connection $A_0 \in \mathcal{A}$.  
	Then we can identify  $\mathcal{A}$ with $ \Omega^1(Z, \mathfrak{g}_Q)$.   The Sobolev norm on   $\mathcal{A}$ is just defined as usual.  For $\mathcal{N}$, we  use $A_0$ to define the Sobolev norm as follows.   Consider an equivariant  embedding  $ \iota: M \hookrightarrow \mathbb{R}^N$ for some positive integer $N$.    For $u \in \mathcal{N}$,   $\iota \circ u$ becomes an  equivariant map  from $Q$ to $\mathbb{R}^N$.  Then the Sobolev norm of $u$ is defined to be the usual Sobolev norm of $\iota \circ u$.   By the same  trick, we can define the Sobolev norm for the gauge group as well.
	For more details, we refer the  reader to Appendix B of \cite{KW}.
	
	We assume that  $kp> dim Z$ throughout. We can extend the definition of   generalised Seiberg--Witten equations (\ref{eq6}), (\ref{eq20}), and   monopoles  to  the Sobolev  completion of  $\mathcal{A} \times \mathcal{N}$. The monopoles are not necessarily smooth by  the definition. But we can always  find a smooth one after a suitable gauge transformation (see  Corollary 5.3.3 of \cite{HS}). 

	\subsection{Summary of the proof} \label{section3}
	As the proof of Theorem \ref{thm0} and Theorem  \ref{thm1} are basically the same, we focus on the 3--dimensional case from now  on.   We will  indicate the corresponding changes for the case of Theorem \ref{thm1} in Section \ref{section6}.

	Here  we summarise  the idea of the proof.    To this end, let us introduce several functions  as follows:   Let $\delta_0>0$ be  the injectivity radius  of $(Y, g_Y)$. Given a monopole $(A, u)$ and $y\in Y$,   for $r \le \delta_0$,  we define
	\begin{equation} \label{eq33}
		\begin{split}
			F_y(r):=\frac{1}{r}\int_{B_r(y)} |d_Au|^2 + 2|\mu(u)|^2  \ \ and \ \ f_y(r):=\int_{\partial B_r(y)} |\chi_0 \circ u|^2.
		\end{split}
	\end{equation}
	The frequency function is defined by $N_y(r): = \frac{r^2F_y(r)}{f_y(r)}$.  The proof of the theorem is based on the following observations:
	\begin{enumerate}
		\item
		The $Weitzenb\ddot{o}ck$  formula leads to a uniform upper  bound on $\rho_0 \circ u$ (see Lemma \ref{lem1}).  In particular, the function $F_y$ is controlled by $N_y$.
		\item
		Under the assumption $\rho_0\circ u \ge c_{\diamondsuit}^{-1}$, we  show that  $F_y$ is  almost monotone with respect to $r$.  The  Heinz trick and the monotonicity  of $F_y$ yields the following: If $F_y$ is sufficiently small for some $r$, then we  obtain  a  bound on $|d_Au|^2+ |F_a|^2$ over a ball $B_r(y)$.  Moreover,  such a  bound only depends on $ c_{\heartsuit}$, $ c_{\diamondsuit}$, $r$, and some  geometric data.    The bound    ensures that we can extract a convergent subsequence of the monopoles.
		\item
		To ensure that $F_y$ is small enough for some $r$,  we show that  $N_y$ also satisfies    certain   monotonicity property.  This property implies that $N_y$ is  controlled by $\frac{1}{\rho_0\circ u}$.  As a consequence, we can get a uniform bound for $|d_Au|^2+ |F_a|^2$ on the region $\{ y\in Y \vert \rho_0 \circ u(y) \ge c_{\diamondsuit}^{-1}\}$.
	\end{enumerate}
	Many computations  and arguments here are modelled  on \cite{TW}, \cite{HW}, and \cite{WZ}.
	
	\begin{remark}
		The proof here cannot be applied to 4--dimensional case  when $dim G \ne 0$.   Apart the more complicated computation, the reason is that the  frequency  function doesn't satisfy  the monotonicity property. (See \cite{T2}.)
	\end{remark}

	\begin{remark}
		In some cases,  the moment map can be written as $\mu=2\rho_0 \nu$, where $\nu: \mathscr{C}(N) \to  \mathfrak{sp}(1)^*\otimes \mathfrak{g}^*$ is a 3--Sasaki moment map on $\mathscr{C}(N) $. If $\inf_{\mathscr{C}(N)}|\nu| \ge c_0^{-1}$, then one can use the same argument in Theorem 5.4.1 of  \cite{HS} to obtain a uniform upper bound on $\rho_0 \circ u$ without the assumption $\int_Y \rho_0 \circ u \le c_{\heartsuit}$. In these cases, the assumption   $\int_Y \rho_0 \circ u \le c_{\heartsuit}$ in Theorem \ref{thm0} can be removed.   An example that fulfils the condition    $\inf_{\mathscr{C}(N)}|\nu| \ge c_0^{-1}$ is  given  on  page 56 of \cite{HS}.
	\end{remark}

	\section{A priori estimates} \label{section4}
	Let   $(A, u) $ be a  generalised monopole.   According to Corollary 5.3.3 of \cite{HS}, $(A, u)$ is gauge equivalent to a smooth monopole. Therefore, we   assume that  the monopoles under consideration are  smooth throughout.
	
	Recall that the Swann bundle satisfies   $\chi_2 =0$. Under this assumption,      keep in mind that we have
	\begin{equation}\label{eq12}
		d_A u =\nabla_A^{TM} \chi_0 \circ u
	\end{equation}
	for any spinor $u\in \mathcal{N}$. This is an important property that helps us to reduce the non--linear differential to a linear one.     Equation (\ref{eq12}) is  proved in Corollary 4.6.2 of \cite{HS}.  Since (\ref{eq12})  plays a crucial  role in the later analysis,   we prove it again in Lemma \ref{lem14}.  We fix a $G$--invariant metric $g_{\mathfrak{g}}$ over the Lie algebra $\mathfrak{g}$ throughout.
	\begin{definition}
		Recall that $K^{M, G}  \vert_{u}$ is the fundamental vector field along $u$.   The $C^k$--norm  of   $K^{M, G}  \vert_{u}$ is defined by   $$ |K^{M, G} \vert_{u}|_{C^k}:=\sup  \sum\limits_{p=0}^k (|(\nabla^M)^p  K^{M, G} \vert_{u} | ,  $$ where $\nabla^M $ is  the  Levi--Civita  connection of $(M, g_M)$. The definition is similar for $K^{M, Sp(1)} \vert_u$.
	\end{definition}
	
	Recall that the moment map is an  equivariant map  $\mu: M \to \mathfrak{sp}(1)^* \otimes \mathfrak{g}^*$.  We define the  Hessian of $\mu$  by the Levi--Civita connection $\nabla^M$ and  the trivial connections on $\mathfrak{sp}(1)^*$ and $\mathfrak{g}^*$, denoted by $Hess \mu$.

	

	\subsection{Some identities}
	Before we estimate the monopole, let us   deduce some  useful identities (Lemmas \ref{lem14}, \ref{lem11}, \ref{lem9}) which are  used for the later computation.
	\begin{lemma} \label{lem14}
		Let $(A, u) \in \mathcal{A} \times  \mathcal{N}$. Then we have $	d_A u =\nabla_A^{TM} \chi_0 \circ u$.
	\end{lemma}
	\begin{proof}
		By Lemma 2.2.29 of \cite{MC1},  $\chi_0 = grad \rho_0$.  By Proposition 5.6 of \cite{AS}, we have    $g_M= Hess \rho_0$. For any $v, w \in TM$, we have
		\begin{equation*}
			\begin{split}
				g_M(\nabla_v^M\chi_0, w) &= v(g_M(\chi_0, w)) - g_M(\chi_0, \nabla^M_v  w) \\
				& = v w \rho_0 - (\nabla^M_v w) \rho_0 \\
				& = (Hess \rho_0) (v, w ) = g_M(v, w).
			\end{split}
		\end{equation*}
		Therefore, $\nabla^M_v \chi_0 =v $ for any $v \in TM$.
		
		Let $v$ be a vector filed on $Y$ and $\tilde{v}$ be  a  horizontal lift of $v$ with respect to $A$. In particular, $A(\tilde{v})=0$.  By definition and the the above observation, we have
		\begin{equation*}
			\begin{split}
				\nabla_{A, v}^{TM} (\chi_0 \circ u) & = \mathcal{K} \circ d_{A, v}^{TM}(\chi_0 \circ u) =\mathcal{K} \circ d(\chi_0 \circ u)(\tilde{v}) \\
				& =\mathcal{K} \circ d \chi_0 \circ du(\tilde{v})\\
				& = \nabla_{du(\tilde{v})}^M \chi_0 =du(\tilde{v}) = d_{A, v} u.
			\end{split}
		\end{equation*}
	\end{proof}

	\begin{lemma} \label{lem11}
		Let $(A,u)\in \mathcal{A} \times \mathcal{N}$ and  $v\in C^{\infty}(Y, TY) $.  We have the following identities:
		\begin{enumerate}
			\mathitem
			\begin{align*}
				\nabla_{A, v}   \mu(u)= d_u \mu (  d_{A, v} u);
			\end{align*}
			
			\mathitem
			\begin{align*}
				\nabla_{A, v} \nabla_{A, w} \mu(u)=Hess \mu \vert_{u} (d_{A, v}u, d_{A, w} u ) +  d_u \mu ( \nabla^{TM}_{A, v} d_{A, w} u) .
			\end{align*}
		\end{enumerate}
	\end{lemma}
	\begin{proof}
		Under the identification $C^{\infty}(Q,  (\mathbb{R}^3)^* \otimes \mathfrak{g})^H \cong C^{\infty}(Y, TY^* \otimes  \mathfrak{g}_Q)$, we have
		\begin{equation*}
			\nabla_{A,  v} \mu(u) \vert_{\pi(p)} = \frac{d}{dt} \mu( u(\gamma(t))) \vert_{t=0} =d_{u(p)}\mu \circ d_pu(\gamma'(0))  =d_{u(p)}\mu ( d_{A,  {v}} u ),
		\end{equation*}
		where $\gamma$ is a  horizontal path in $Q$ such that $\gamma(0)=p$ and $\gamma'(0)=\tilde{v}(p)$ and  $\tilde{v}$ is a horizontal lift of $v$ with respect to $A$.   Then we get the first assertion of the lemma.

		To prove the second statement,  let  $\sigma$ be a  horizontal path in $Q$ such that $\sigma(0)=p$ and $\sigma'(0)=\tilde{v}(p)$ as before,  where   $\tilde{v}$ is a  horizontal lift of $v$ with respect to $A$. Let $d_{A, w}$ denote a vector field on $M$ such that $d_{A,w } \vert_{u(\sigma(t))} =d_{A,w} u(\sigma(t))$.   Then
		\begin{equation}\label{eq7}
			\begin{split}
				&(Hess\mu \vert_{u(p)})( d_{A, w} u,  d_{A, v} u )\\
				=& \frac{d}{dt}d_{u(\sigma(t))}\mu(d_{A,w}  u(\sigma(t))) \vert_{t=0} - d_{u(p)} \mu (\nabla^M_{d_{A,v} u} d_{A,w} )\\
				=& d_{(u(p), d_{A,w}u \vert_p) }(d\mu)(d_{A,v}^{TM}(d_{A, w} u)   )- d_{u(p)} \mu ( \mathcal{K} \circ \frac{d}{dt} d_{A,w}  {u(\sigma(t))} \vert_{t=0} )\\
				=& d_{(u(p), d_{A,w}u \vert_p) }(d\mu)(d_{A,v}^{TM}(d_{A, w} u)   )- d_{u(p)} \mu (\nabla^{TM}_{A, v} d_{A, w}u).\\
			\end{split}
		\end{equation}
		
		On the other hand, we have
		\begin{equation} \label{eq8}
			\begin{split}
				\nabla_{A,v}\nabla_{A, w}  \mu(u) \vert_{\pi(p)}& = \frac{d}{dt} d_{u(\sigma(t))}  \mu( d_{A, {v}}u({{\sigma(t)}}))  \vert_{t=0}\\
				& =  d_{(u(p), d_{A, w} u \vert_p)} (d\mu) ( {d}^{TM}_{A, v} {d}_{  A, w} u   ).
			\end{split}
		\end{equation}
		Combine Equations (\ref{eq7}) and  (\ref{eq8}); then we get the second conclusion.
	\end{proof}
	
	Fix $u \in \mathcal{N}$. Define an operator $\mathcal{Y}_u: \Omega^1(Y, \mathfrak{g}_Q) \to \Gamma(u^*TM) $ along $u$ by  $$\mathcal{Y}_u(\eta) :=\sum_k I_k K^{M, G}_{<\eta, e_k>} \vert_u,$$ where  $\{e_k\}_{k=1}^3$ is an orthonormal basis of $TY^*$ and $\eta \in \Omega^1(Y, \mathfrak{g}_Q)$.

	\begin{lemma} \label{lem9}
		Let $(A, u) \in \mathcal{A} \times \mathcal{N}$. For any $\xi \in  C^{\infty}(Q, \mathfrak{g} )^H$, we have the following identities:
		\begin{enumerate}
			\item
			If $(A, u)$ is a monopole, then $$ <d_A^* F_a, \xi > = -g_M(K^{M, G}_{\xi} \vert_u, d_A u);$$
			\item
			$< \mathcal{Y}_u( \nabla_A \mu(u)) ,d_A u>= -| d_u\mu  (d_A u)|^2.$
		\end{enumerate}
	\end{lemma}
	\begin{proof}
		Let $\{x^i\}_{i=0}^3$ be the normal coordinates  at $y \in Y$   and $\xi \in  C^{\infty}(Q, \mathfrak{g} )^H$.  We may assume that $d_a \xi =0$ at $y$.
		Write $<\mu(u), \xi> =\sum\limits_{i=1}^3 b_i dx^i$. Then
		\begin{equation*}
			\begin{split}
				<*d_A \mu(u), \xi>= (b_{2;1} -b_{1;2})dx^3+ (b_{3;2} -b_{2;3})dx^1+  (b_{1;3} -b_{3;1})dx^2.
			\end{split}
		\end{equation*}
		By Lemma \ref{lem11} and the Dirac equation (\ref{eq4}), we have
		\begin{equation} \label{eq22}
			\begin{split}
				b_{2;1 }-b_{1;2} &=<d_u \mu(d_{A, 1} u), dx^2 \otimes \xi>_{g_Y \otimes g_{\mathfrak{g}}}-<d_u \mu(d_{A, 2} u), dx^1 \otimes \xi>_{g_Y \otimes g_{\mathfrak{g}}}\\
				&= \omega_2(K^{M, G}_{\xi} \vert_u, d_{A, 1} u ) -\omega_1(K^{M, G}_{\xi} \vert_u, d_{A, 2} u )\\
				&=g_M(K^{M, G}_{\xi} \vert_u, I_2d_{A, 1} u  - I_1 d_{A, 2} u)  =g_M(I_3K^{M, G}_{\xi}\vert_u, -I_1d_{A, 1} u  - I_2 d_{A, 2} u)\\
				&=g_M(K^{M, G}_{\xi} \vert_u, d_{A, 3}u).
			\end{split}
		\end{equation}
		The computation is similar for the other two terms; the details are left to the reader.  Therefore, we have
		$$<d_A^* F_a, \xi > =- \sum_i  g_M(K^{M, G}_{\xi} \vert_u, d_{A, i} u) dx^i=-g_M(K^{M, G}_{\xi} \vert_u, d_{A } u). $$

		Now we prove the second equality.  By definition and Lemma \ref{lem11}, we have
		\begin{equation*}
			\begin{split}
				&< \mathcal{Y}_u( \nabla_A \mu(u)) ,d_A u>=  <I_k K^{M, G}_{<\nabla_{A, j} \mu(u), dx^k>}, d_{A, j} u>\\
				=& -\omega_k (K^{M, G}_{<d_u \mu (d_{A, j} u ), dx^k> }, d_{A, j} u)\\
				=& -g_{\mathfrak{g}} ( < d_u\mu(d_{A, j} u), dx^k>, < d_u\mu(d_{A, j} u), dx^k> )=-|d_u\mu(d_Au)|^2.
			\end{split}
		\end{equation*}
		
	\end{proof}

	\subsection{Estimates}
	Now we begin to estimate the monopoles. From now on, we use $c_0>0$ to denote a   large uniform constant which only depends on the uniform bound $\int_Y \rho_0 \circ u \le c_{\heartsuit}$ and the geometric data. 
	It may be  different from line to line.

	First of all, let us recall the   $Weitzenb\ddot{o}ck$  formula in our setting. 
	\begin{theorem}[ Theorem 6.2.1 of \cite{MC1}]
		The $Weitzenb\ddot{o}ck$  formula in dimension three is as follows:
		\begin{equation} \label{eq26}
			\begin{split}
				D_A^{lin, u^*} D_A u = \nabla_A^{TM, *} d_A  u+ \frac{s_Y}{4} \chi_0 \circ u + \mathcal{Y}_u(*F_a ),
			\end{split}
		\end{equation}
		where $D_A^{lin, u^*}$ and $\nabla_A^{TM, *} $ are  the $L^2$--adjoint of the linearization  (\ref{eq28}) and $\nabla_A^{TM}$ respectively, and $s_Y$ is the scalar curvature of $(Y, g_Y)$. 
	\end{theorem}
	
	\begin{lemma} \label{lem1}
		Let    $(A, u)$ be a  monopole.   Assume that $\int_Y |\chi_0 \circ u|^2 \le c_{\heartsuit} $. Then  we have $\int_Y |d_A u|^2  + 2|\mu (u)|^2 \le c_0$ and $|\chi_0 \circ u|^2 \le c_0$.
	\end{lemma}
	\begin{proof}
		By  Equations (\ref{eq6}) and (\ref{eq26}), we have
		\begin{equation} \label{eq19}
			\begin{split}
				0= \nabla_A^{TM, *} \nabla_A^{TM} \chi_0 \circ u+ \frac{s_Y}{4} \chi_0 \circ u + \mathcal{Y}_u(*F_a ).
			\end{split}
		\end{equation}
		According to Proposition 3.2.6 of \cite{MC},  $<\mu,  \zeta \otimes \xi> =-\frac{1}{2}\omega_{\zeta} (\chi_0 , K^{M, G}_{\xi})$. Hence,
		\begin{equation}
			\begin{split}
				< \mathcal{Y}_u(*F_a), \chi_0 \circ u  >=&   <\chi_0 \circ u, I_k K^{M, G}_{<*F_a , e_k>} \vert_u>\\
				=& \omega_k (\chi_0 \circ u,   K^{M, G}_{<*F_a , e_k>} \vert_u )\\
				=& -2 <\mu(u), \zeta_k \otimes <*F_a , e_k> >  \\
				=& 2 g_ {\mathfrak{g}}(<\mu(u), e_k> , <-*F_a  , e_k> )=  2 |\mu(u)|^2.
			\end{split}
		\end{equation}
		Take inner product  of Equation (\ref{eq19}) with $\chi_0 \circ u $; then  we have
		\begin{equation} \label{eq5}
			\begin{split}
				\frac{1}{2}d^*d |\chi_0 \circ u|^2 + |d_Au |^2 + \frac{s_Y}{4}|\chi_0 \circ u|^2 + 2|\mu(u)|^2=0.
			\end{split}
		\end{equation}
		Integrating    (\ref{eq5})   we   obtain the bound   $\int_Y |d_A u|^2  + 2|\mu (u)|^2 \le c_0$.
		
		Assume that $|\chi_0 \circ u |^2 $ attains  its maximum at $p$. Let $G_p$ be the Green function of $ d^*d$ with pole at $p$. Then we have
		\begin{equation*}
			\begin{split}
				|\chi_0 \circ u|^2(p) & \le   c_0 \int_{B_r(p)}  G_p|\chi_0\circ u|^2   +     c_0 \int_{Y-B_r(p)}  G_p|\chi_0\circ u|^2 +  \frac{1}{vol(Y)}\int_{Y} |\chi_0 \circ u|^2 \\
				&  \le c_0 r^2 |\chi_0 \circ u|^2(p) + (\frac{c_0}{r } +  \frac{1}{vol(Y)}) \int_Y |\chi_0 \circ u|^2.
			\end{split}
		\end{equation*}
		Take  $r>0$ such that $c_0r^2 =\frac{1}{2}$; then we get the  bound on the sup--norm of $|\chi_0 \circ u|$.
	\end{proof}

	\begin{remark}\label{rmk1}
		Let   $Y'$  be the open  submanifold  in Theorem \ref{thm0}  such that  $\inf\limits_{Y'}\rho_0 \circ u_n  \ge {c_{\diamondsuit}}^{-1}$.    According to the  assumption that $N$ is compact,  (\ref{eq15}) and   Lemma \ref{lem1}, the condition   $\inf\limits_{Y'}\rho_0 \circ u_n  \ge c_{\diamondsuit}^{-1}$ implies that the images  of $\{u_n \vert_{Y'}\}_{n=1}^{\infty}$ are contained in a compact subset of $M$.
		
		The consequent of the above observation is that   we can find constants $c_k>0$ satisfying 
		\begin{equation} \label{eq27}
			|Hess\mu \vert_{u_n \vert_{Y'}}|_{C^k}+|Rm_M \vert_{u_n \vert_{Y'}}|_{C^k}+ |K^{M, G}  \vert_{u_n \vert_{Y'}} |_{C^k}  + |K^{M,Sp(1)}  \vert_{u_n \vert_{Y'}} |_{C^k}   \le c_k,
		\end{equation}
		where $Rm_M$ is the Riemannian curvature of $g_M$.  As tensors on $M$, the norms for  $Hess \mu$ and  $Rm_M$   are the usual tenor norms defined by the metric $g_M$ and the  Levi--Civita  connection. If we restrict $Hess \mu$  and $Rm_M$ on a  compact subset $K$ of  $M$, then we can get the uniform bound in (\ref{eq27}). (The bound depends on the compact set $K$.)
		Note that  the fundamental vector field can be regarded as a map $K^{M,G} : M \to TM \otimes \mathfrak{g}^*$. Hence, if we restrict  $K^{M,G}  $ to a compact set   of  $M$, then we can get a $C^k$--bound  on $K^{M,G}  $ for any fixed $k$ as well.  It is similar for  $K^{M,Sp(1)}$.
	\end{remark}
	
	\textbf{ In the rest part of this paper, we assume that $\int_Y |\chi_0 \circ u|^2 \le c_{\heartsuit}$; unless otherwise stated.}
	\begin{lemma} \label{lem12}
		Let    $(A, u)$ be a  monopole. We have the following identity:
		\begin{equation*}
			\begin{split}
				&\frac{1}{2}d^*d |F_a|^2 + |\nabla_{A} (*F_a)|^2 +  |\mathcal{Y}_u(*F_a )|^2 \\
				=& -\frac{s_Y}{4} g_M(\chi_0 \circ u, \mathcal{Y}_u(*F_a )) -  <tr_{g_Y}Hess  \mu \vert_u (d_{A}u, d_{A} u), \mu(u)> .
			\end{split}
		\end{equation*}
		In particular, we have
		\begin{equation*}
			\begin{split}
				\frac{1}{2}d^*d |F_a|^2 + |\nabla_A (*F_a)|^2 +  \frac{1}{2} |\mathcal{Y}_u(*F_a )|^2  \le c_0|\chi_0 \circ u|^2 + |Hess  \mu \vert_u | |d_A u|^2 |F_a|.
			\end{split}
		\end{equation*}
	\end{lemma}
	\begin{proof}
		Let $\{x^i\}_{i=0}^3$ be the  normal coordinates  at $y$. Then by Lemma \ref{lem11} and the $Weitzenb\ddot{o}ck$  formula (\ref{eq19}), we have
		\begin{equation} \label{eq1}
			\begin{split}
				\nabla_{A, i}\nabla_{A, i} \mu(u) \vert_y
				&= Hess \mu \vert_u(d_{A, i} u, d_{A, i} u) + d_u \mu(\nabla_{ A,   i}^{ TM} d_{ A, i} u) \\  
				&= Hess  \mu \vert_{u} (d_{A, i} u, d_{A, i} u) +d_u \mu( \frac{s_Y}{4} \chi_0 \circ u + \mathcal{Y}_u(*F_a)).    \\
			\end{split}
		\end{equation}
		For any vector  field $v \in C^{\infty}(Q, u^* TM)^H$, we have
		\begin{equation} \label{eq17}
			\begin{split}
				&<d_u \mu (v), *F_a >_{g_Y \otimes g_{\mathfrak{g}}} \\
				=&< d_u \mu(v), \zeta_l \otimes <*F_a , dx^l>> \\
				=&\omega_{ l} (K^{M, G}_{<*F_a, dx^l>} \vert_u, v)\\
				=&- g_M(I_l K^{M, G}_{<*F_a , dx^l>} \vert_u, v)  = -g_M( \mathcal{Y}_u(*F_a ), v).
			\end{split}
		\end{equation}
		Take inner product of Equation (\ref{eq1}) with $*F_a$. Then Equation (\ref{eq17}) implies that
		\begin{equation*}
			\begin{split}
				&<\nabla_{A}^*\nabla_{A}(*F_a),  *F_a >\\
				=&<Hess \mu \vert_{u}(d_{A, i}u, d_{A, i} u), *F_a> - |\mathcal{Y}_u(*F_a )|^2 -\frac{s_Y}{4} g_M(\chi_0 \circ u, \mathcal{Y}_u(*F_a)).
			\end{split}
		\end{equation*}

	\end{proof}

	\begin{lemma} \label{lem4}
		Let $(A, u)$ be a monopole. Then we have
		\begin{equation} \label{eq2}
			\begin{split}
				&\frac{1}{2} d^*d |d_A u|^2 + |\nabla_{A}^{TM} d_Au|^2  + |d_u\mu (d_Au)|^2 + |d_A^* F_a|^2 \\
				\le&  c_0|F_a|^2  +( |Rm_{M} \vert_u|  +   |K^{M,G} \vert_u|^2_{C^1} ) |d_Au|^4   + c_0 (1+ |K^{M, Sp(1)}|^4_{C^1}).
			\end{split}
		\end{equation}
	\end{lemma}
	\begin{proof}
		By the curvature formulas in   Lemmas 2.4.1 and 2.4.2 of \cite{HS}, we have
		\begin{equation*}
			\begin{split}
				&\nabla_{A, j}^{TM} \nabla^{TM}_{A, i} d_{A, i}u \\
				=&  \nabla_{A, i}^{TM} \nabla^{TM}_{A, j} d_{A, i}u + \mathcal{K} \circ K^{TM, H}_{F_A(\partial_j, \partial_i)} \vert_{d_{A, i} u} + Rm_M (d_{A, j}u, d_{A, i}u   ) d_{A, i}u\\
				=&  \nabla_{A, i}^{TM} (\nabla^{TM}_{A, i} d_{A, j}u  + K^{M, H}_{F_A(\partial_j, \partial_i) } \vert_u ) + \mathcal{K} \circ K^{TM, H}_{F_A(\partial_j, \partial_i)} \vert_{d_{A, i} u} + Rm_M (d_{A, j}u, d_{A, i}u   ) d_{A, i}u\\
				=& -  \nabla_{A}^{TM, *} \nabla^{TM}_{A} d_{A, j}u +  \nabla_{A, i}^{TM} ( K^{M, H}_{F_A(\partial_j, \partial_i) } \vert_u )\\
				+& \mathcal{K} \circ K^{TM, H}_{F_A(\partial_j, \partial_i)} \vert_{d_{A, i} u} + Rm_M (d_{A, j}u, d_{A, i}u   ) d_{A, i}u\\
			\end{split}
		\end{equation*}
		Lemma 2.4.3  of \cite{HS} implies that   $\mathcal{K} \circ K^{TM, H}_{F_A(\partial_j, \partial_i)} \vert_{d_{A, i} u} =  \nabla^M_{d_{A, i} u}( K^{M, H}_{F_A(\partial_j, \partial_i) } \vert_u ) $.   By definition, we have
		\begin{equation*}
			\begin{split}
				\nabla_{A, i}^{TM} ( K^{M, H}_{F_A(\partial_j, \partial_i) } \vert_u ) &=  \nabla^M _{d_{A, i} u} (K^{M, H}_{F_A(\partial_j, \partial_i)} \vert_u)  +  K^{M, H}_{\nabla_{A, i} F_A(\partial_j, \partial_i)} \vert_u   \\
				& = \nabla^M _{d_{A, i} u} (K^{M, H}_{F_A(\partial_j, \partial_i)} \vert_u)  + K^{M, Sp(1)}_{\nabla_{A_{Y, i}} F_{A_Y}(\partial_j, \partial_i)} +   K^{M, G}_{\nabla_{A, i} F_{a}(\partial_j, \partial_i )} \\
				& = \nabla^M _{d_{A, i} u} (K^{M, H}_{F_A(\partial_j, \partial_i)} \vert_u)  + K^{M, Sp(1)}_{d_{A_Y}^* F_{A_Y}(\partial_j)} +   K^{M, G}_{ d_A^*F_{a}(\partial_j)}  \vert_u.
			\end{split}
		\end{equation*}
		Therefore, 
		\begin{equation*}
			\begin{split}
				&\nabla_{A}^{TM, *} \nabla^{TM}_{A} d_{A, j}u = \nabla^{TM}_{A, j} \nabla_{A}^{TM, *}  d_{A}u  + 2 \nabla^M_{d_{A, i} u} (K^{M, H}_{F_A(\partial_j, \partial_i)} \vert_u)\\
				&+  K^{M, Sp(1)}_ {d_{A_Y}^* F_{A_Y}(\partial_j)} +   K^{M, G}_{ d_A^*F_{a}(\partial_j)}  \vert_u+ Rm_{M} (d_{A, j}u, d_{A, i}u   ) d_{A, i}u. \\
			\end{split}
		\end{equation*}
		By the  $Weitzenb\ddot{o}ck$  formula, we get
		\begin{equation*}
			\begin{split}
				\nabla_{A}^{TM, *} \nabla^{TM}_{A} d_{A, j}u &= -\frac{1}{4} (s_Y)_j  \chi_0 \circ u - \frac{s_Y}{4} d_{A, j} u  -\nabla_{A, j}^{TM} \mathcal{Y}_u(*F_a) \\
				&+ 2 \nabla^M_{d_{A, i} u} (K^{M, H}_{F_A(\partial_j, \partial_i)} \vert_u)+ Rm_M (d_{A, j}u, d_{A, i}u   ) d_{A, i}u \\
				&+  K^{M, Sp(1)}_{d^*_{A_Y} F_{A_Y}(  \partial_j)} \vert_u+   K^{M, G}_{ d_A^*F_{a}(\partial_i)}  \vert_u \\
				& = -\frac{1}{4} (s_Y)_j  \chi_0 \circ u - \frac{s_Y}{4} d_{A, j} u  - ( \nabla^M_{d_{A, i} u } \mathcal{Y}_u)(*F_a) - \mathcal{Y}_u (\nabla_{A, j} *F_a)\\
				&+ 2 \nabla^M_{d_{A, i} u} (K^{M, H}_{F_A(\partial_j, \partial_i)} \vert_u)+ Rm_M (d_{A, j}u, d_{A, i}u   ) d_{A, i}u \\
				&+  K^{M, Sp(1)}_{  d_{A_Y}^*F_{A_Y}(\partial_j)} \vert_u +   K^{M, G}_{ d_A^*F_{a}(\partial_j)}  \vert_u \\
			\end{split}
		\end{equation*}
		Take inner product of the above equation with $d_Au $. By Lemma \ref{lem9},   we get
		\begin{equation*}
			\begin{split}
				&\frac{1}{2} d^*d |d_A u|^2 + |\nabla_{A}^{TM} d_Au|^2 + |d_u\mu (d_Au)|^2 + |d_A^* F_a|^2\\
				\le&   c_0(1+ |\nabla^M K_{F_{A_Y}}^{M, Sp(1)} \vert_u|^2 + |K^{M, Sp(1)}_{ d^*_{A_Y}F_{A_Y}}|^2  )|d_Au|^2 \\
				&+ c_0|F_a|^2  +(|Rm_{M} \vert_u|   + |\nabla^M K^{M,G} \vert_u|^2 ) |d_Au|^4 +  c_0|\chi_0 \circ u|^2 +c_0.
			\end{split}
		\end{equation*}
		
	\end{proof}

	\begin{lemma}  \label{lem3}
		Let    $(A, u)$ be a  monopole.  Recall that  the  function $F_y$ is defined by  $F_y(r) =r^{-1}\left(\int_{B_r(y)} |d_Au|^2 + 2|\mu(u)|^2\right)$ for $r \le \delta_0$, where $\delta_0 $ is the  injectivity radius of $(Y, g_Y)$.  Then we have
		\begin{equation}
			\begin{split}
				\frac{\partial F_y}{\partial r}   \ge \frac{2}{r } \int_{\partial B_r(y)}(|d_{A ,r } u|^2 +  |F_a(\partial_r, \cdot )|^2) - c_0( |K^{M, Sp(1)}_{F_{A_Y}} \vert_u|_{C^0} + 1)^2  F_y(r) - c_0 r^2.
			\end{split}
		\end{equation}
		Moreover, if $   |K^{M, Sp(1)}_{F_{A_Y}} \vert_u|_{C^0 }     \le c_0$, then there exists constants $c_1, c_2 >0$ (independent of $y$) such that  $e^{c_1 r}F_y(r) + c_2 r^3$ is non--decreasing.
	\end{lemma}
	\begin{proof}
		By the  definition of $F$, we  have
		\begin{equation}
			\begin{split}
				\frac{\partial F}{\partial r} =-\frac{1}{r^2} \int_{B_r} ( |d_A u|^2 + 2|\mu(u) |^2)  + \frac{1}{r} \int_{\partial B_r} ( |d_A u|^2 +2 |\mu(u)|^2).
			\end{split}
		\end{equation}
		
		Define symmetric 2--tensors $S= S_{ij} dx^i \otimes dx^j $  and $R=R_{ij} dx^i\otimes dx^j$ respectively  by
		\begin{equation}
			\begin{split}
				&S_{ij}:=   < d_{A, i} u, d_{A, j}u >  -\frac{1}{2}(g_Y)_{ij} |d_A u|^2,\\
				&R_{ij}:= <(F_a)_{ik},    (F_a)_{jk}   > - (g_Y)_{ij}   |F_a|^2.
			\end{split}
		\end{equation}
		Let $\{x^i\}$ be the normal coordinates at $y$. The divergence of $S$ at $y$ is
		\begin{equation} \label{eq9}
			\begin{split}
				S_{ij; j} &=   <\nabla^{TM}_{A, j} d_{A, i} u, d_{A, j} u> +    < d_{A, i} u,  \nabla^{TM}_{A, j}  d_{A, j} u>  - \delta_{ij}  <\nabla^{TM}_{A, j} d_{A, k} u, d_{A, k} u>\\
				&=  <K^{M, H}_{F_A(\partial_j, \partial_i) } \vert_u , d_{A, j} u>   - <\nabla_A^{TM*} d_A u,  d_{A, i}u>  \\
				&= <K^{M, H}_{F_A(\partial_j, \partial_i) } \vert_u , d_{A, j} u>   + < \frac{s_Y}{4} \chi_0 \circ u + \mathcal{Y}_u(*F_a),  d_{A, i}u>.  \\
			\end{split}
		\end{equation}
		Note that
		\begin{equation}\label{eq10}
			\begin{split}
				< \mathcal{Y}_u(*F_a),  d_{A, i}u> &=<I_k K^{M, G}_{(*F_a)_k}, d_{A, i} u> \\
				&=-\omega_k(K^{M, G}_{(*F_a)_k}, d_{A, i} u) =-<d_u\mu(d_{A, i} u ), dx^k \otimes (*F_a)_k>\\
				&=-< \nabla_{A, i} \mu(u), *F_a> = \frac{1}{2}\partial_i|F_a|^2.
			\end{split}
		\end{equation}
		
		To compute $R_{ij;  j}$, note that $R_{ij} = -< (*F_a)_i, (*F_a)_j >$ for $i \ne j$. By a  direct  computation, we get
		\begin{equation*}
			\begin{split}
				R_{1j; j} &=- \partial_1 |F_a|^2  + \partial_1|(F_a)_{1k}|^2 + R_{12;2} +R_{13;3}\\
				&=- \frac{1}{2}\partial_1 |F_a|^2 + <(*F_a)_j, (*F_a)_{j; 1} -  (*F_a)_{1; j} > -<(*F_a)_{j;j}, (*F_a)_1>.
			\end{split}
		\end{equation*}
		Since $ d_A^*(*F_a)=-(*F_a)_{j;j}=0$, the last term vanishes.
		For  any $\xi \in C^{\infty}(Q, \mathfrak{g})^H$, by the same computation as in (\ref{eq22}), we have
		$$<(*F_a)_{2;1} -(*F_a)_{1; 2},  (*F_a)_2> =g_M(  K^{M,G}_{(*F_a)_2},  -  d_{A ,3 } u ) =    g_M(  K^{M,G}_{F_a(\partial_1, \partial_3)},   d_{A ,3 } u ).  $$ The computation is similar for the other  terms. In sum, we have
		\begin{equation}\label{eq11}
			\begin{split}
				R_{ij;j } = -\frac{1}{2}\partial_i |F_a|^2 + g_M(K^{M, G}_{F_a(\partial_i, \partial_j)}, d_{A, j} u).
			\end{split}
		\end{equation}
		Let $T=S+R$. Combine Equations (\ref{eq9}), (\ref{eq10}) and (\ref{eq11}); then  we  have
		\begin{equation}
			\begin{split}
				&T_{ij;j } = <K^{M, Sp(1)}_{F_{A_Y}(\partial_j, \partial_i)} \vert_u, d_{A, j} u >  + <\frac{s_Y}{4} \chi_0 \circ u, d_{A_i} u> \\
				\mbox{and \ }&|div T| \le c_0 (|\chi_0\circ u| + |K^{M, Sp(1)}_{F_{A_Y}} \vert_u|) |d_Au|.
			\end{split}
		\end{equation}
		
		
		Let $r(p) =dist (p, y)$ denote  the distance function. The divergence  theorem implies that
		\begin{equation*}
			\begin{split}
				\int_{B_r}  div T ( \frac{1}{2} \nabla r^2)=r \int_{\partial B_r} T_{ij} r_i r_j - \int_{B_r} <T, \frac{1}{2} \nabla^2 r^2>.
			\end{split}
		\end{equation*}
		Note that $\frac{1}{2}\nabla^2 r^2=g_Y + O(r)$.  Hence
		\begin{equation*}
			\begin{split}
				\int_{B_r}  div T ( \frac{1}{2} \nabla r^2)&=r \int_{\partial B_r}  T( \partial_r,  \partial_r) - \int_{B_r}  tr_{g_Y} T + O(r) \int_{B_r} |T|\\
				&= r\int_{\partial B_r} \left( (|d_{A, r} u|^2  + |F_a(\partial_r, \cdot)|^2)- \frac{1}{2} (|d_A u|^2  +2 |\mu(u)|^2)) \right) \\
				&+ \frac{1}{2}\int_{B_r} (|d_Au|^2  + 2|\mu(u)|^2) + O(r) \int_{B_r} ( |d_Au |^2 +  2|\mu(u)|^2).
			\end{split}
		\end{equation*}
		Therefore, we obtain
		\begin{equation}  
			\begin{split}
				\frac{\partial F}{\partial r}& =  -\frac{2}{r^2}\int_{B_r(x)}  div T  (\frac{1}{2} \nabla r^2) + \frac{2}{r } \int_{\partial B_r}(|d_{A ,r } u|^2 +  |F_a(\partial_r, \cdot)|^2) +O(1) F(r)  \\
				& \ge \frac{2}{r } \int_{\partial B_r}(|d_{A ,r } u|^2 + |F_a(\partial_r, \cdot )|^2) - c_0( |K^{M, Sp(1)}_{F_{A_Y}} \vert_u|_{C^0} + 1)^2  F(r) - c_0 r^2.
			\end{split}
		\end{equation}
		
		Assume that   $|K_{F_{A_Y}}^{M, Sp(1)} \vert_u|_{C^0} \le c_0$.  Then   we find constants $c_1, c_1'>0 $ such that
		\begin{equation} \label{eq29}
			\frac{\partial}{\partial r} F(r) \ge -c_1F(r)-c_1' r^2.
		\end{equation}
		By the differential inequality (\ref{eq29}), we have
		\begin{equation*}
			\begin{split}
				(e^{c_1r}F(r)) ' &\ge -c_1' e^{c_1 r} r^2 \\
				&\ge -c_1' e^{c_1 \delta_0} r^2 =(-\frac{1}{3}e^{c_1 \delta_0} c_1' r^3)'.  \\
			\end{split}
		\end{equation*}
		Therefore,  $(e^{c_1r} F + \frac{1}{3}e^{c_1 \delta_0} c_1' r^3)' \ge 0$. Take $c_2=  \frac{1}{3}e^{c_1 \delta_0} c_1'$; then    $e^{c_1r} F +   c_2r^3$ is non--decreasing.

	\end{proof}

	\begin{prop}[Heinz trick, see Appendix A of \cite{TW}] \label{prop1}
		Let $U \subset Y$ be an open subset and $f: U \to [0, \infty)$ be a non--negative function. Suppose that there exists a constant $c >0$ such that  $f$ satisfies the following properties:
		\begin{enumerate}
			\item
			$d^*d f \le c  (f ^2+1)$.
			\item
			If $B_s (y) \subset B_{\frac{r}{2}}(x) $, then  $ s^{-1} \int_{B_s(y)}  f \le c  r^{-1} \int_{B_r(x)} f + c  r^2$.
		\end{enumerate}
		Then there exist constants $c_0>0$,  $ \epsilon_0>0$ and $\delta_1>0$ depending on the $c $ and the geometric data such that for all $ r\le \delta_1$ and $B_r(x) \subset U$ with
		\begin{equation*}
			\begin{split}
				\epsilon=r^{-1} \int_{B_r(x)} f \le \epsilon_0,
			\end{split}
		\end{equation*}
		we have
		\begin{equation*}
			\begin{split}
				\sup\limits_{B_{\frac{r}{4}}(x)} f \le c_0 r^{-2} \epsilon + c_0 r^2.
			\end{split}
		\end{equation*}
	\end{prop}
	To simplify the notation, we assume that $\delta_1= \delta_0$   all time.
	
	\textbf{From now on,  we assume that}
	\begin{equation} \label{eq24}
		|Hess \mu \vert_u|_{C^0} + |Rm_M \vert_{u}|_{C^0}+ |K^{M, Sp(1)} \vert_{u}|_{C^1} + |K^{M, G} \vert_{u}|_{C^1}  \le c_0.
	\end{equation}
	As pointed out in Remark \ref{rmk1}, (\ref{eq24}) is true over the submanifold $Y'$.
	\begin{corollary} \label{lem10}
		There exists a  constant $\epsilon_0>0$ depending only on the geometric data and the bound $c_0$ in   (\ref{eq24})  with the following significant: Let  $(A, u) $ be  a  monopole  satisfying (\ref{eq24}). Fix   $y \in Y$. If $F_y( r) \le \epsilon_0 $, then  we have
		\begin{equation*}
			\begin{split}
				|d_A u|^2 +2 |F_a|^2 \le  c_0 r^{-2} \epsilon_0 + c_0 r^2
			\end{split}
		\end{equation*}
		over the ball $B_{\frac{r}{4}} (y)$. 
	\end{corollary}
	\begin{proof}
		Take $f=|d_Au |^2 + 2|F_a|^2$ in Proposition \ref{prop1}. The first condition of  Proposition \ref{prop1} follows from Lemmas \ref{lem1}, \ref{lem12}, and \ref{lem4}.
		To verify  the second condition, note that
		$B_s(y) \subset B_{\frac{r}{2}}(x)$ implies that $B_\frac{r}{2}(y) \subset B_r(x)$.  By Lemma \ref{lem3}, we have
		\begin{equation*}
			\begin{split}
				s^{-1} \int_{B_s(y)} f &\le e^{c_1s} F_y(s) + c_2s^3 \\
				& \le e^{c_1{\frac{\delta_0}{2}}} F_y(\frac{r}{2}) + c_2 \frac{r^3}{8}\\
				& \le 2e^{c_1{\frac{\delta_0}{2}}} \left( r^{-1}\int_{B_r(x)} f \right)
				+ c_2 \frac{r^3}{8}.
			\end{split}
		\end{equation*}
		Thus,  we can apply  Proposition \ref{prop1} and then obtain  the result.
	\end{proof}

	\section{Frequency function } \label{section5}
	In this subsection, we follow the techniques  in \cite{HW}  and \cite{WZ} to show that  a sequence of monopoles $\{(A_n, u_n)\}_{n=1}^{\infty}$ has a convergent subsequence if the assumption (\ref{eq24}) holds.  

	\begin{definition} \label{def1}
		Let $\epsilon_0>0$ be the constant given by Corollary \ref{lem10}. For any $x \in Y$, define a number at $x$ by
		\begin{equation}
			\begin{split}
				r(x):= \sup\{ r \in (0, \delta_0] \vert  \frac{1}{r } \int_{B_r(x)} (|d_A u|^2 +2 |F_a|^2)  \le \epsilon_0\}.
			\end{split}
		\end{equation}
	\end{definition}
	Corollary  \ref{lem10} implies that for any $r \le r(x)$, we have a uniform upper bound on  $|d_A u|^2 +|F_a|^2$ over a ball $B_{\frac{r}{4}}(x)$.

	Let $(A,u)$ be a monopole. Reintroduce the functions $f_x, F_x$ defined in (\ref{eq33}) for  $(A, u)$. The   frequency function is $N_x(r)= \frac{r^2F_x(r)}{f_x(r)}$.  We study their properties  in this section.
	\begin{lemma} \label{lem13}
		The function  $f_x(r) = \int_{\partial B_r(x)} |\chi_0 \circ u|^2$  satisfies the  following equation:
		\begin{equation*}
			\begin{split}
				f_x'(r) =  \frac{2}{r} f_x(r)  +   2r  F_x(r) +  2\int_{B_r(x) } \left(\frac{s_Y}{4} |\chi_0 \circ u|^2  \right) + \mathfrak{r}_x(r),
			\end{split}
		\end{equation*}
		where $|\mathfrak{r}_x(r)|\le c_0r  f_x(r).$
	\end{lemma}
	\begin{proof}
		By a straightforward  computation, we have
		\begin{equation*}
			\begin{split}
				f'(r) = \frac{2}{r} f(r)  + \int_{ \partial B_r} \partial_r | \chi_0 \circ  u|^2 + \mathfrak{r}(r).
			\end{split}
		\end{equation*}
		The term  $\mathfrak{r}(r) $ comes from the non--flatness of the metric. It  satisfies $|\mathfrak{r}(r)| \le c_0 r f(r)$.   By the divergence theorem, we obtain
		\begin{equation*}
			\begin{split}
				\int_{ \partial B_r} \partial_r | \chi_0 \circ u|^2 =   \int_{  B_r}  div \nabla (| \chi_0 \circ  u|^2) =  -\int_{  B_r}  d^*d| \chi_0 \circ  u|^2.  
			\end{split}
		\end{equation*}
		According to  Equation (\ref{eq5}), we get  
		\begin{equation*}
			\begin{split}
				\int_{ \partial B_r} \partial_r | \chi_0 \circ u |^2 =2r F +  2\int_{B_r } \frac{s_Y}{4} |\chi_0 \circ u|^2.
			\end{split}
		\end{equation*}

	\end{proof}

	\begin{corollary} \label{lem8}
		For any $0< s<r \le \delta_0$, we have
		\begin{equation} \label{e1}
			\begin{split}
				e^{c_0s^2 } \frac{f(s)}{s^2} \le e^{c_0 r^2} \frac{f(r)}{r^2}.
			\end{split}
		\end{equation}
		Also, we have  $\int_{B_r} |\chi_0 \circ u|^2 \le c_0 r f(r)$.
	\end{corollary}
	\begin{proof}
		By the fact that $\int_{B_r} h^2 \le c_0(r^2 \int_{B_r} |dh|^2 + r \int_{\partial B_r}  h^2 )$, Kato's inequality, and    Lemma \ref{lem13}, we have
		\begin{equation} \label{eq13}
			\begin{split}
				f' \ge \frac{2}{r}f - c_0 rf .
			\end{split}
		\end{equation}
		Then we get   (\ref{e1}) by integrating  Inequality (\ref{eq13}).  Using this monotonicity  property, we have
		\begin{equation*}
			\begin{split}
				\int_{B_r} |\chi_0 \circ u|^2 =\int_0^r f(s)ds \le c_0 \frac{f(r)}{r^2} \int_0^r s^2 ds \le c_0r f(r).
			\end{split}
		\end{equation*}
	\end{proof}
	Define a non--negative function $\kappa_x$  by the relation $\kappa_x^2= e^{-2\sigma_x} r^{-2} f_x$, where $$\sigma_x (r): =  \int_{0}^r  \frac{1}{f_x(s)} (\int_{B_s(x)} \left( \frac{s_Y}{4} |\chi_0 \circ u|^2   \right)+ \frac{1}{2}\mathfrak{r}_x(s) ) ds.$$
	Note that Corollary \ref{lem8} implies that $|\sigma(r)| \le c_0 r^2$ and $|\sigma'(r)| \le c_0 r$. 
	By Lemma \ref{lem13}, it is straightforward to check that  $\kappa$ satisfies  the  differential equality
	\begin{equation} \label{eq14}
		\begin{split}
			\frac{d \kappa}{dr} =\frac{1}{r} N \kappa.
		\end{split}
	\end{equation}
	
	\begin{lemma} \label{lem5}
		We have the following inequality:
		\begin{equation*}
			\begin{split}
				N'(r)
				&\ge \frac{2e^{-2\sigma}}{r^2 \kappa^2} \int_{\partial B_r}\left( |d_{A ,r } u   -\frac{N}{ r} \chi_0 \circ u|^2  +  |F_a(\partial_r, \cdot)|^2 \right)     \\
				&-c_0(   |K^{M, Sp(1)}_{F_{A_Y}} \vert_u|_{C^0}^2  +  1) N - c_0 \frac{r^2 e^{-2\sigma}}{\kappa^2}.
			\end{split}
		\end{equation*}
	\end{lemma}
	\begin{proof}
		By Equation (\ref{eq14}) and Lemma \ref{lem3}, we have
		\begin{equation} \label{eq16}
			\begin{split}
				N'(r)  &=\frac{F'}{\kappa^2}e^{-2\sigma} -2N \sigma' - 2 \frac{N^2}{r}\\
				&\ge \frac{2e^{-2\sigma}}{r  \kappa^2} \int_{\partial B_r}(|d_{A ,r } u|^2  +  |F_a(\partial_r, \cdot)|^2 )  -c_0(  |K^{M, Sp(1)}_{F_{A_Y}} \vert_u|_{C^0}  +   1)^2  N \\
				&- c_0 \frac{r^2 e^{-2\sigma}}{\kappa^2} -2Nc_0 r-2\frac{N^2}{r}.
			\end{split}
		\end{equation}
		Using  integration by part and the   $Weitzenb\ddot{o}ck$ formula, we get
		\begin{equation*}
			\begin{split}
				N&= \frac{e^{-2\sigma}}{r \kappa^2}\int_{B_r} ( |d_A u|^2 +2 |\mu(u)|^2) \\
				&= \frac{e^{-2\sigma}}{r \kappa^2}\int_{B_r}  2 |\mu(u)|^2  +  \frac{e^{-2\sigma}}{r \kappa^2}\int_{\partial B_r} <\chi_0\circ u, d_{A, r} u> \\
				&-    \frac{e^{-2\sigma}}{r \kappa^2} \int_{B_r}( \frac{s_Y}{4}|\chi_0 \circ u |^2  + < \chi_0 \circ u, \mathcal{Y}_u (*F_a)>)  \\
				&=\frac{e^{-2\sigma}}{r\kappa^2}\int_{\partial B_r} <\chi_0 \circ u, d_{A, r} u> + O(r^2).
			\end{split}
		\end{equation*}
		
		The term $-2\frac{N^2}{r}$ in Inequality (\ref{eq16}) is equal  to $-4\frac{N^2}{r}+ 2\frac{N^2}{r}$.  Replace    $-4\frac{N^2}{r}$ by  $ -4 \frac{e^{-2\sigma}}{r\kappa^2}\int_{\partial B_r} <\frac{N}{r}\chi_0 \circ u, d_{A, r} u> + O(r) N.$ Then we get
		\begin{equation*}
			\begin{split}
				N'(r)
				&\ge \frac{2e^{-2\sigma}}{r^2 \kappa^2} \int_{\partial B_r} \left(|d_{A ,r } u  -\frac{N}{ r} \chi_0 \circ u|^2   + 2|F_a(\partial_r, \cdot)|^2 \right) \\
				&- c_0(  |K^{M, Sp(1)}_{F_{A_Y}} \vert_u|_{C^0}  +   1)^2  N - c_0 \frac{r^2 e^{-2\sigma}}{\kappa^2}  -c_0 r N.
			\end{split}
		\end{equation*}
		
	\end{proof}


	\begin{lemma} \label{lem7}
		Let $x \in Y $ such that $\rho_0\circ u(x) \ne 0$.    For any $\epsilon_1 >0$,    there exists  a constant  $c_{\epsilon_1}>0 $ such that   we have  $N_x(r) \le \epsilon_1$ for any $r \le \min\{\sqrt{\epsilon_1}, c_{\epsilon_1} \rho_0 \circ u(x)\}$.
	\end{lemma}
	\begin{proof}
		Let $0<r \le  r_0  $ and $\rho=  \rho_0 \circ u(x)$. By  Corollary \ref{lem8}, we know that $ \kappa^2(r) \ge c_0\rho$. Also, $\kappa^2 \le c_0$ because of Lemma \ref{lem1}. By (\ref{eq14}),  we have
		\begin{equation} \label{eq30}
			\int_r^{r_0} 2\frac{N(t)}{t} dt  = \log \frac{\kappa^2(r_0)}{ \kappa^2(r)} \le \log (c_0 \rho^{-1}).
		\end{equation}
		By Lemma \ref{lem5} and $\kappa^2 \ge c_0\rho$, we have
		\begin{equation} \label{eq31}
			N(t) \ge c_0^{-1} N(r) - c_0\rho^{-1}r_0^2,
		\end{equation}
		for $r \le t\le r_0$.
		Combine (\ref{eq30}) and (\ref{eq31}); then we get
		\begin{equation} \label{eq3}
			\begin{split}
				N(r) \le c_0 \frac{\log (c_0 \rho^{-1}  )}{ \log (\frac{r_0}{r})} +  c_0r_0^2 \rho^{-1}.
			\end{split}
		\end{equation}
		If $c_0\rho^{-1} \le 1$, then we  can take $r_0 = \sqrt{\epsilon_1}$.   If $c_0 \rho^{-1}  >1$, we   take $r_0$ to be square root  of    $\frac{\epsilon_1}{2c_0} \rho $.  Then the right hand side of (\ref{eq3}) is  less than $\epsilon_1$ if  $c_0 \frac{\log c_0 \rho^{-1}}{ \log \frac{r_0^2}{r^2}} \le \frac{1}{4} \epsilon_1$.  Note this is equivalent to  require that  $r\le  c_{\epsilon_1} \rho$.
		
		In sum, if $r \le \min\{\sqrt{\epsilon_1}, c_{\epsilon_1}\rho\}$, then  $N(r) \le \epsilon_1$.
	\end{proof}

	\begin{corollary} \label{lem6}
		Let $x \in Y $ such that $\rho_0\circ u(x) \ne 0$. Then the  quantity  $r(x)$ defined in Definition \ref{def1}  satisfies $r(x) \ge \min\{\sqrt{\epsilon_1}, c_{\epsilon_1} \rho_0 \circ u(x)\} $.
	\end{corollary}
	\begin{proof}
		By the uniform bound on $|\chi_0\circ u|$ in Lemma \ref{lem1}, we have  $F_x(r) =N_x(r) \frac{f}{r^2} \le c_0 N_x(r)$. Take $\epsilon_1= \frac{\epsilon_0}{c_0}$. Then Lemma \ref{lem7} implies that $r(x) \ge \min\{\sqrt{\epsilon_1}, c_{\epsilon_1} \rho_0 \circ u (x)\} $.
	\end{proof}

	\begin{proof}[Proof of  Theorem \ref{thm0}]
		Let $\{(A_n, u_n)\}^{\infty}_{n=1}$ be the sequence of monopoles in Theorem \ref{thm0}.   After gauge transformations, we can assume that  $\{(A_n, u_n)\}^{\infty}_{n=1}$  are smooth.  Note that the functions $\rho_0  \circ u_n  $  are preserved under the gauge transformations.
		
		Suppose that there is an open submanifold $Y'$   such that $ {\uplim\limits_{n\to \infty}}\inf\limits_{Y'}\rho_0 \circ u_n \vert_{Y'} \ge c_{\diamondsuit}^{-1}$ for some constant $c_{\diamondsuit}>0$. 
		After passing to  a subsequence, this condition implies that the images of $\{u_n \vert_{Y'}\}_{n=1}^{\infty}$ are contained in a compact subset of $M$. Therefore, we have  $$| Hess\mu { \vert_{u_n\vert_{Y'}}}| + |Rm_M \vert_{u_n\vert_{Y'}} | + |K^{M,G}  \vert_{u_n \vert_{Y'}}|_{C^1} + |K^{M,Sp(1)}  \vert_{u_n \vert_{Y'}}|_{C^1}\le c_0.$$ (See Remark \ref{rmk1}.) As Corollary \ref{lem10} and the analysis in Section \ref{section5} are local, we can apply them to the equations over  $Y'$.
		
		For any compact subset $K \subset Y'$,   Lemma \ref{lem7} implies that $r_n(x) \ge c_K^{-1}$ over $K$, where $r_n(x)$ is the  $u_n$--version of  number defined in Definition \ref{def1}. Then Corollary \ref{lem10} deduces a uniform bound $|d_{A_n} u_n|(x) + |F_{a_n}|(x) \le C_K$ for any $x\in K$.
		
		
		By Theorem  B and the patching argument  of \cite{KW},  for any $p >3$, we can find a sequence of gauge transformations $g_n \in \mathcal{G}^{2, p}$ such that  the sequence   $\{g_n^* A_n -A_0 \}$ has a uniform bound on the  $W^{1, p}$--norm and $\{g_n^* A_n \}$  converges  in $W^{1,p}$  weakly  to $A \in \mathcal{A}^{1,p}$.  Write $g_n^*A_n= A_0 + \alpha_n$. Here we take $A_0$ to be a smooth approximation of $A$.  By Theorem 8.1 of \cite{KW},   we can assume that $d_{A_0}^*\alpha_n=0$ and $*\alpha_n \vert_{\partial K}=0$ after   gauge transformations.
		By the  Sobolev embedding theorem, we have  $|g_n^* A_n -A_0| \le c_0$.
		
		By Lemma 3.6.10 of \cite{MC},  $d_{g_n^*A_n} (g_n^{-1} u_n) \vert_{p}=d_{A_n} u_n \vert_{pg(p)}$. Thus $|d_{g_n^*A_n} (g_n^{-1} u_n)|$ are  still uniformly bounded. Hence, we have
		\begin{equation} \label{eq21}
			\begin{split}
				|d_{A_0} (g_n^{-1}u_n)| \le |d_{g_n^*A_n} (g_n^{-1} u_n) | +| K^{M,G}_{g_n^* A_n -A_0} \vert_{u_n}| \le c_0.
			\end{split}
		\end{equation}
		By a direct computation, we have
		\begin{equation}\label{eq25}
			\begin{split}
				d_{A_0}^*d_{A_0} \alpha_n =d_{g_n^*A_n}^* F_{g_n^*a_n} -d_{A_0}^* F_{a_0} - \frac{1}{2}d_{A_0}(*[\alpha_n \wedge \alpha_n]) - \frac{1}{2}*[\alpha_n \wedge F_{g_n^*a_n}].
			\end{split}
		\end{equation}
		Let $K'\subset K$ be a   slightly smaller  compact subset. Using Equation (\ref{eq25}), Lemma \ref{lem9},   and the Sobolev  embedding theorem, we have
		\begin{equation*}
			\begin{split}
				| \alpha_n|_{W^{2,p}(K')} &\le  c_0 |d_{A_0}^*d_{A_0} \alpha_n|_{L^p(K)} + c_0 |\alpha_n|_{W^{1,p}(K)}\\
				&\le  c_0 |d_{g_n^*A_n}^*F_{g_n^*a_n}|_{L^p(K)} + |\alpha_n|_{L^{p}(K)} |F_{a_n}|_{L^{\infty}(K)} \\
				&+ |\alpha_n|_{W^{1,p}(K)} |\alpha_n|_{L^{\infty}(K)} + c_0 |\alpha_n|_{W^{1,p}(K)}+ c_0 \\
				&\le  c_0 |K^{M, G} \vert_{u_n}|_{L^{\infty}(K)} |d_{A_n} u_n|_{L^{\infty}(K)} \\
				& + |\alpha_n|_{L^{p}(K)} |F_{a_n}|_{L^{\infty}(K)} +c_0 |\alpha_n|_{W^{1,p}(K)}   + c_0\\
				& \le c_0, \\
			\end{split}
		\end{equation*}
		The constant $c_0$ here also depends on $A_0$ and $K'$. By the Rellich's lemma, $\{g_n^*A_n\}_{n=1}^{\infty}$ converges to  $A$ strongly in the  $W^{1, p}$--norm over $K'$.

		Inequality  (\ref{eq21})  implies that  $\{(g_n^*A_n, g_n^{-1}u_n)\}_{n=1}^{\infty}$ converges weakly in $W^{1,p}$--norm  and converges in $C^0$ to  $(A, u)$.  The pair $(A, u)$ satisfies the generalised Seiberg--Witten equations over $K'$.  After a gauge transformation, we assume that $(A, u)$ is smooth in a slightly smaller compact subset (still call it $K'$). (See Corollary 5.3.3 of \cite{HS}.) For sufficiently large $n$, we can write $$g_n \cdot (A_n, u_n) =(A+ a_n, exp_u(v_n)),$$ where $(a_n, v_n) \in \Omega^1(Y, \mathfrak{g}_Q)^{2,p} \oplus W^{1,p}(Y, \pi_!u^*TM)$. Moreover,
		$ \{(a_n, v_n)\}_{n=1}^{\infty} $  converges in  $C^0$ to zero, $\{a_n \}_{n=1}^{\infty} $ converges in $W^{1,p}$ to zero and  $|\nabla_A^{TM} v_n | \le c_0$.
		
		After   gauge transformations,  we can further assume that $d_A ^* a_n =0$ and $*a_n \vert_{\partial K'}=0$  and the above estimates are still true. Therefore, the  generalised Seiberg--Witten equations   can be written as
		\begin{equation*}
			\begin{split}
				&D_A^{lin, u} v_n = Q_0(\nabla_A^{TM} v_n, v_n)+ Q_1(d_Au, v_n) +Q_2(v_n , a_n),\\
				&(d_A^*+d_A) a_n =*d_u\mu(v_n )  + Q_3(v_n, v_n) +Q_4(a_n, a_n),
			\end{split}
		\end{equation*}
		where $\{Q_i\}_{i=0}^4$ are certain  quadratic operators in the sense that $$|Q_i(x, y)|_{W^{k,p}} \le c_0 |x|_{W^{k,p}} |y|_{W^{k, p}}$$ for $kp>3$ or $k=0$ and $p=\infty.$ The standard elliptic bootstrapping argument implies that $(a_n, v_n)$ converges in  $C^{\infty}$ to zero over a  slightly smaller compact subset. (See Theorem 5.3.2 of \cite{HS}.)
	\end{proof}
	
	\section{4--dimensional  case} \label{section6}
	First of all, observe that  the 4--dimensional generalised Seiberg--Witten equations (\ref{eq20}) are  reduced to the Dirac equation  $D_Au =0$ when $dim G=0$.  We still have   the $Weitzenb\ddot{o}ck$ formula in dimension four. (See \cite{HS}.) Moreover, it becomes
	\begin{equation} \label{eq32}
		\begin{split}
			D_A^{lin, u^*} D_A u = \nabla_A^{TM, *} d_A  u+ \frac{s_X}{4} \chi_0 \circ u
		\end{split}
	\end{equation}
	when $dim G=0$.  Using the  $Weitzenb\ddot{o}ck$ formula (\ref{eq32}) and the assumption $\int_X \rho_0 \circ u \le c_{\heartsuit}$, one can deduce a  uniform bound on  $\rho_0 \circ u$ as in Lemma \ref{lem1}.  The functions $F_x$ and $N_x$ are replaced by
	\begin{equation*}
		\begin{split}
			F_x(r) =\frac{1}{r^2}\int_{B_r(x)}|d_Au|^2   \mbox{\ and \ } N_x(r) =\frac{r^3F_x(r)}{f_x(r)}.
		\end{split}
	\end{equation*}
	The number in Definition \ref{def1} at $x \in X$ is replaced  by  $$r(x)= \sup\{ r \in (0, \delta_0] \vert  \frac{1}{r^2 } \int_{B_r(x)} |d_A u|^2  \le \epsilon_0\}.$$ 
	The same argument can show that the functions $N_x(r)$ and $F_r(x)$ satisfy the monotonicity properties in  Lemmas \ref{lem3}, \ref{lem5}.   The computation now is simpler  due to  the lack of   curvature terms.

	
	\subsection*{Acknowledgment}
	This work is carried out at the   University of Adelaide, the author acknowledges the comfortable environment provided by the School of Mathematics and  the support from ARC Discovery project grant DP170101054 with Chief Investigators Mathai Varghese and David Baraglia.
	
	The author would like to thank Dr.Varun Thakre's helpful comments.  He also wants to thank the
	anonymous referee, whose comments and suggestions have greatly improved this paper.
	
 
 E-mail address: davidtree33@protonmail.com

\end{document}